\documentclass[10pt,a4paper]{article}
\usepackage[T1]{fontenc}
\usepackage[left=2cm, right=2cm, top=3cm, bottom=3cm]{geometry}
\usepackage{xcolor}
\usepackage{amssymb,amsthm,amsmath,dsfont}
\usepackage{mathrsfs}
\usepackage{natbib,hyperref}
\usepackage{authblk}
\usepackage{graphicx}
\usepackage{subcaption}
\usepackage{float} 

\usepackage[linesnumbered,ruled,vlined]{algorithm2e}
\newcommand{\scr}[1]{\mathscr{#1}}
\newcommand{\bb}[1]{\mathbb{#1}}
\newcommand{\al}[1]{\mathcal{#1}}
\newcommand{\rrm}[1]{\mathrm{#1}}
\newcommand{\I}{\mathds{1}}
\newtheorem{theorem}{Theorem}[section]

\newtheorem{proposition}[theorem]{Proposition}
\newtheorem{lemma}[theorem]{Lemma}

\newtheorem{remark}[theorem]{Remark}
\newtheorem{assumption}[theorem]{Assumption}

\providecommand{\keywords}[1]{%
  \vspace{1ex}%
  \noindent\hangindent=2.5em%
  \hangafter=0%
  \noindent\textbf{Keywords: }#1%
}

\title{Randomised Euler-Maruyama method for SDEs with H\"older continuous drift coefficient driven by $\alpha$-stable L\'evy process}
\author[a,1]{Jianhai Bao}
\author[b,2]{Haitao Wang}
\author[b,3]{Yue Wu}
\author[b,4]{Danqi Zhuang}
\affil[a]{Center for Applied Mathematics, Tianjin University, 300072 Tianjin, P.R. China}
\affil[b]{Department of Mathematics and Statistics, University of Strathclyde, Glasgow, G1 1XH, UK} \affil[1]{\href{mailto:jianhaibao@tju.edu.cn}{jianhaibao@tju.edu.cn}}
\affil[2]{\href{mailto:haitao.wang@strath.ac.uk}{haitao.wang@strath.ac.uk}}
\affil[3]{\href{mailto:yue.wu@strath.ac.uk}{yue.wu@strath.ac.uk}}
\affil[4]{\href{mailto:danqi.zhuang.2021@uni.strath.ac.uk}{danqi.zhuang.2021@uni.strath.ac.uk}}
\date{}
\begin{document}
\maketitle
\begin{abstract}
      In this paper, we examine the performance of randomised Euler-Maruyama (EM) method for additive time-inhomogeneous SDEs with an irregular drift driven by symmetric $\alpha$-table process, $\alpha\in (1,2)$. In particular, the drift is assumed to be $\beta$-H\"older continuous in time and bounded $\eta$-H\"older continuous in space with $\beta,\eta\in (0,1]$. The strong order of convergence of the randomised EM in $L^p$-norm is shown to be $1/2+(\beta \wedge (\eta/\alpha)\wedge(1/2))-\varepsilon$ for an arbitrary $\varepsilon\in (0,1/2)$, higher than the one of standard EM, which cannot exceed $\beta$. The result for the case of $\alpha \in (1,2)$ extends the almost optimal order of convergence of randomised EM obtained in \cite{bao2025randomisedeulermaruyamamethodsdes} for SDEs driven by Gaussian noise ($\alpha=2$), and coincides with the performance of EM method in simulating time-homogenous SDEs driven by $\alpha$-stable process considered in \cite{biswas2024explicit}. Various experiments are presented to validate the theoretical performance.
\end{abstract}

\section{Introduction}\label{intro}

There have been increasingly focused on stable L\'evy processes due to their ability to capture complex real-world behaviours that lie beyond the scope of classical Gaussian models. By allowing for heavy-tailed distributions and discontinuous sample paths with jumps, these processes generalize Brownian motion and provide more realistic representations of phenomena involving impulsive noise, discontinuities, or infinite variance \cite{levy,MR1406564,MR1739520}. However, this growing interest comes with notable challenges: the lack of closed-form density expressions (except for special cases like Cauchy and L\'evy) \cite{levy}; non-locality introduced by jumps, etc. Those challenges hinder progress in understanding towards stochastic process involving $\alpha$-stable L\'evy process.  For example, the comprehensive picture is recently completed \cite{Sivaoleg2020,zhang2021,distributionhuang,MR2945756,MR3445537,MR3795063,MR4235987,wu2021wellposednessdensitydependentsde} of the strong wellposedness of additive SDE 
	\begin{equation}\label{main SDE}
    \begin{cases}
			\rrm d X_t=b(t,X_t)\rrm dt+\rrm dL_t ,&\\
			X_0=x_0\in\bb R^d, &
		\end{cases}
	\end{equation}
	where  $L=(L_t)_{0\le t\le T}$ is a $d$-dimensional $\alpha$-stable L\'evy process 
 with $\alpha\in (0,2]$ defined on the probability space $(\Omega_L, \scr F^L,\bb P_L)$, and the drift coefficient $b:[0,T]\times \mathbb{R}^d\to \mathbb{R}^d$ is H\"older in space, while its Gaussian counterpart has been exhaustively studied in the last century \cite{MR568986,MR336813}. 
        
 Denote by $C_b^{\beta, \eta}([0,T]\times\bb R^d;\bb R^d)$ the space of functions being $\beta$-H\"older continuous in time and bounded $\eta$-H\"older continuous in space with $\beta,\eta\in (0,1]$, equipped with the norm 
	\begin{equation}\label{Cbetanorm}
		\begin{aligned}
			\left\|f\right\|_{C_b^{\beta, \eta}}=\sup_{t\in [0,T],x\in \mathbb{R}^d}|f(t,x)|+\sup_{x\in \mathbb{R}^d,s\ne t}\frac{|f(t,x)-f(s,x)|}{|t-s|^\beta}+
			\sup_{t\in [0,T],x\ne y}\frac{|f(t,x)-f(t,y)|}{|x-y|^\eta} 
		\end{aligned}
	\end{equation} 
    and
    by $C_b^{\eta}(\bb R^d;\bb R^d)$ the space of functions being bounded $\eta$-H\"older continuous in space with $\eta\in (0,1]$, equipped with the norm 
	\begin{equation*}
		\begin{aligned}
			\left\|f\right\|_{C_b^{ \eta}}=\sup_{x\in \mathbb{R}^d}|f(x)|+
			\sup_{x\ne y \in \mathbb{R}^d}\frac{|f(x)-f(y)|}{|x-y|^\eta}. 
		\end{aligned}
	\end{equation*} 
    In this paper, given the existence and pathwise uniqueness \cite{chenandzhanglevy} within the range $2\eta+\alpha>2$, we consider the numerical simulation to the additive time-inhomogeneous SDE \eqref{main SDE}, where $b\in C_b^{\beta, \eta}([0,T]\times\bb R^d;\bb R^d)$ and $\alpha \in (1,2]$.  For simplicity, let $T=1$, and use $C_b^{\beta,\eta}$ for short.

   For time-homogenous SDE \eqref{main SDE} with $b\in  C_b^{\eta}(\bb R^d,\bb R^d)$, the widely studied scheme is the Euler-Maruyama method (EM) on a fixed stepsize $1/n$, $n\in \mathbb{N}$, with the numerical solution given by
\begin{equation}\label{eqn:emcontinuous}
    \bar{X}^{(n)}_t=x_0+\int_0^t b\big(\bar{X}^{(n)}_{\kappa_n(s)}\big)\,\mathrm{d} s  +{L_t},
\end{equation}
where $\kappa_n(s):=\lfloor   ns \rfloor /n $ with $\lfloor   a\rfloor $ representing the largest integer that does not exceed $a$. Thanks to stochastic sewing lemma (SSL) \cite{Khoa2020} and its conditional shifted version \cite[Lemma 3.1]{butkovsky2024strongrateconvergenceeuler}, the order of convergence of EM method has been lifted to $(\eta/\alpha)\wedge (1/2)+1/2-\varepsilon$, for arbitrary $\varepsilon\in (0,1/2)$, compared to the previous attempts \cite{BaohuangEM,BaohuangzhangEM,HuangEMHOLDER,HangSuoEM,distributionhuang,EMlevyLi}. SSL sharpens estimates by systematically decomposing stochastic processes into a martingale component and a bounded variation process, enabling tighter control of residual terms through their orthogonal structure. Till now, SSL has shaped the understanding towards the numerical convergence for boundary scenarios of time-homogenous SDEs for Gaussian setting when drifts are H\"older \cite{MR4344136,ling2024strongEM}.   However, as argued in \cite{bao2025randomisedeulermaruyamamethodsdes}, for time-dependent SDE \eqref{main SDE}, the order of convergence of EM method can not go beyond $\beta$. We therefore follow a different approach that considers \emph{randomised} method.

Taking $\bar{Y}^{(n)}_t=\bar{X}^{(n)}_t-L_t$, one may treat it as the Euler simulation for the random ODE of the form 
\begin{equation}\label{eqn:randomODE1}
    \dot Y_t=b(Y_t+L_t):=\tilde b_L(t, Y_t), \qquad Y_0=x_0.
\end{equation}
where the random source $L_t$ determines the time irregularity of $\tilde b_L$. 
Note that when $\eta=1$,  random ODE \eqref{eqn:randomODE1} can be realised pathwisely with a unique pathwise solution, i.e., for a given and frozen realisation of $(L_t)_{t\in [0,1]}$,  random ODE \eqref{eqn:randomODE1} is of Carath\'eodory type and its solution exists uniquely. On the one hand, the Euler method,  however, fails to converge to the exact solution under this setting (see \cite[Section 2.3]{montecarlo} and \cite[Section 1.1]{Krusewu2017}), and SSL loses its power here as the randomness is frozen. On the other hand, incorporating extra randomness into the existing Euler scheme helps: instead of  using the grid point $\kappa_n(s)$, evaluating $\tilde b_L(s,\cdot)$ at a uniformly sampled point between two consecutive grid points $\kappa_n(s)$ and $\kappa_n(s)+1/n$.  The so-called randomised Euler method \cite{Krusewu2017} has been shown to converge to the exact solution with an order $1/2$ in $L_p$-sense\footnote{This error is measured in $L_p$-norm with respect to the extra randomness introduced to the scheme.} and $1/2-\varepsilon$ in almost-sure sense, regardless of the form of the $\alpha$-stable process $L_t$ is.
	
 Due to its ability to handle time-irregularity of equations, randomised methods have been applied to stochastic differential equations (SDEs) \cite{MR3986241,randomizedmilstein},  SDEs with L\'evy diffusion \cite{Pawel2024}, stochastic delay differential equation \cite{MR3880714,randomforsdde} and Mckean-Valsov equations \cite{biswas2024explicit}. It's worth noticing the drifts considered therein are all assumed to be Lipscthiz continuous in spatial variable. With the rising of SSL in handling spatial-irregularity, almost optimal order of strong convergence is obtained for simulating SDE \eqref{main SDE} with H\"older continuous drift subject to standard Brownian motion, via the randomised Euler-Maruyama method \cite{bao2025randomisedeulermaruyamamethodsdes}. In this paper, we examine the performance of randomised Euler-Maruyama method in solving SDE \eqref{main SDE}.
 
 Consider randomised Euler-Maruyama method, let $(\tau_i)_{i\in\bb N}$ be an i.i.d family of the uniform distribution on the interval $(0,1)$ random variables on filtered probability space $(\Omega_{\tau}, \scr F^{\tau}, (\scr F^{\tau}_i)_{i\in \bb N},
	\bb P_{\tau})$, where $\scr F_i^{\tau}$ is the $\sigma$-algebra generated by
	$\{\tau_1,\ldots, \tau_i\}$. We denote a new probability space $(\Omega, \scr F,
	\bb P):=(\Omega_L\times\Omega_{\tau}, \scr F^L\otimes\scr F^\tau,\bb P_L\otimes\bb P_\tau)$, and we can define a new discrete-time filtration $(\scr F_i)_{i\in \{1,\cdots,n\}}$ as $$\scr F_i:=\scr F_{i/n}^L\otimes\scr F^\tau_{i+1} \qquad \text{for any }i\in\{1,\cdots,n\}, $$ 
	where $\scr F_{i/n}^L$ is the  $\sigma$-algebra generated by $L_{i/n}$, and let $t_i:=i/n$ for $i\in \{1,\cdots,n\}$. Then we consider the randomised Euler-Maruyama approximation of SDE \eqref{main SDE} given by 
	\begin{equation}\label{dis EM SDE}
		 \begin{cases}			X^{(n)}_{t_i}&=X^{(n)}_{t_{i-1}}+b(t_{i-1}+\tau_i/n,X^{(n)}_{t_{i-1}})+\Delta^n_iL,\\
			X^{(n)}_0&=x_0,
		\end{cases}
	\end{equation}
	where $\Delta^n_iL:=L_{t_i}-L_{t_{i-1}}$. Next, we consider the continuous version of filtration $(\scr F_t)_{t\in[0,1]}$ denoted by $$\scr F_t:=\scr F_{t}^L\otimes\scr F^\tau_{\lfloor tn\rfloor+1}\qquad \text{for any }t\in[0,1],$$ where $\lfloor s\rfloor$ represents the largest integer that does not exceed $s$. Next, we denote $\kappa_n(s):=\lfloor ns\rfloor/n$ and $\kappa^\tau_n(s):=\lfloor ns\rfloor/n+\tau_{\lfloor ns\rfloor+1}$ and consider the continuous randomised Euler-Maruyama approximation
	\begin{equation}\label{con EM SDE}
		\begin{aligned}
			X^{(n)}_t=x_0+\int_{0}^{t}b(\kappa^\tau_n(s),X^{(n)}_{\kappa_n(s)})\rrm ds+L_t
		\end{aligned}
	\end{equation}

     With the randomised EM defined in \eqref{con EM SDE}, the strong order of convergence can be lifted to $1/2+\gamma-\varepsilon$ for arbitrary $\varepsilon\in (0,1/2)$, where $\gamma=\beta \wedge (\eta/\alpha) \wedge (1/2)$ within the range $$D_{\alpha,\beta, \eta}:=\{(\alpha,\beta, \eta): 2\eta+\alpha>2, (\beta+1)\alpha+\eta>2, \alpha \in (1,2), \beta, \eta \in (0,1]\},$$ see Theorem \ref{theo2.3} under the assumptions set in Section \ref{NMR}.  Our contributions are four-fold:
\begin{itemize}
    \item It is the first time applying randomised scheme in solving SDEs driven by $\alpha$-stable process, $ \alpha \in (1,2)$;
    \item Theorem \ref{theo2.3} for $\alpha \in (1,2)$ extends the work on the strong convergence of randomised EM \cite{bao2025randomisedeulermaruyamamethodsdes} for SDEs driven by Gaussian noise ($\alpha=2$);
    \item  Section \ref{sec:mainthm} extends the PDE-based error analysis of \cite{Pamen2017} from truncated to full symmetric $\alpha$-stable processes;
    \item Numerical benefits are validated via various examples in Section \ref{sec:numerical}.
\end{itemize}
Notations and priori estimations can be found in Section \ref{NMR} and \ref{sec:priori} respectively.
\section{Notation and Main Result}\label{NMR}
	In this section we explain the notation that is used throughout this paper and give the main result in this paper.
	\subsection{Notations}
	Let $\scr B(\bb R^d_0)$ be the Borel $\sigma$-algebra on $\bb R^d_0:=\bb R^d\setminus \{0\}$ and set $\nabla=\big(\frac{\partial}{\partial x_1},\cdots,\frac{\partial}{\partial x_d}\big)^\top$. In the following we introduce some space of function: 
	\begin{itemize}
		\item $C_b^\rho(\bb R^d;\bb R^m)$, $\rho\in(0,1)$ denotes the set of $\bb R^m$-valued bounded functions defined on $\bb R^d$, which are $\rho$-H\"older continuous function, equiped with the norm $$\left\|f\right\|_{C_b^{\rho}}=\sup_{x\in \mathbb{R}^d}|f(x)|+
		\sup_{x\ne y}\frac{|f(x)-f(y)|}{|x-y|^\rho},\quad f\in C_b^\rho(\bb R^d) .$$
		\item $C_b^{k+\rho}(\bb R^d;\bb R^m)$, $k\in \bb N,\ \rho\in(0,1]$ denotes the set of $\bb R^m$-valued bounded functions whose weak derivatives of order $0,1,\cdots, k$ all have representatives belonging to $\rho$-H\"older continuous function, equipped with the norm $$\left\|f\right\|_{C_b^{k+\rho}}=\sum_{i=0}^{k}\sup_{x\in \mathbb{R}^d}|\nabla^i f(x)|+
		\sup_{x\ne y}\frac{|\nabla^kf(x)-\nabla^kf(y)|}{|x-y|^\rho},\quad f\in C_b^{k+\rho}(\bb R^d) .$$
		\item $C^0_b(\bb R^d)$ denotes the space of all bounded (not necessarily continuous) Borel measurable functions in $\bb R^d$ equipped with the supremum norm
		$$\left\| f\right\| _{C^0_b}=\sup_{x\in\bb R^d}\left| f(x)\right| ,\quad f\in C^0_b(\bb R^d). $$
		\item $C([s,t];C^\beta_b(\bb R^d;\bb R^d)),\ \beta>0$ with $s<t$, the collection of $C^\beta_b(\bb R^d;\bb R^d)$-valued continuous
		function over time interval $[s,t]$, equipped with the norm
		$$\left\|f \right\|_{[s,t]}=\sup_{s \le u \le t}\left\|f(u,\cdot) \right\|_{C_b^\beta} . $$
		Note that $C([s,t];C_b^{\beta}(\bb R^d;\bb R^d))$ coincides with $C_b^{0,\beta}([s,t]\times \mathbb{R}^d; \mathbb{R}^d)$ mentioned in Section \ref{intro}, where the former one will be adopted in the later analysis.
		\item $C^\infty_c(\bb R^d)$ denotes the space of all infinitely differentiable $\bb R$-valued functions with compact	support contained in $\bb R^d$.
		\item $C_b^{\beta, \eta}([s,t]\times\bb R^d;\bb R^d)$ with $\beta,\eta\in (0,1]$, is the space of all continuous functions $g:[s,t]\times \mathbb{R}^d\to \mathbb{R}^d$ which are $\alpha$-H\"older continuous in time and bounded $\beta$-H\"older continuous in space, equipped with the norm defined in equation \eqref{Cbetanorm}.
	\end{itemize}
	\begin{remark}
		Let $g\in C_b^{\beta, \eta}([s,t]\times\bb R^d;\bb R^d)$. For any $u\in[s,t]$ we have the estimate:
		\begin{equation}\label{gprop}
			\begin{aligned}
				|g(u,x)-g(u,y)|&=|g(u,x)-g(u,y)|\big(\I_{\{|x-y|\le 1\}}+ \I_{\{|x-y|\ge 1\}}\big)
				\\&\le \|g\|_{C_b^{\beta, \eta}}|x-y|^\eta\I_{\{|x-y|\le 1\}}+2\|g\|_{C_b^{\beta, \eta}}\I_{\{|x-y|\ge 1\}}\\&\le 2\|g\|_{C_b^{\beta, \eta}}\big(|x-y|^\eta\wedge 1\big).
			\end{aligned}
		\end{equation}
	\end{remark}
	For any function $f:\bb R^d\to\bb R^d$, we may represent it as the row vector $f=(f_i)_{i\in[d]},$ where $[d]:=\{1,\ldots,d\}$.
	For a random variable $\xi$, a sub-$\sigma$-algebra $\scr G\subset\scr F$, and $p\ge1$ we introduce the quantity
	
	$$\left\|\xi \right\|_{L^p(\Omega)|\scr G}:=\big(\bb E[|\xi|^p|\scr G]\big)^{\frac{1}{p}}, $$
	which is a $\scr G$-measurable non-negative random variable. It is clear that
	$$\left\|\xi \right\|_{L^p(\Omega)}=\left\| \left\|\xi \right\|_{L^p(\Omega)|\scr G} \right\| _{L^p(\Omega)}. $$
	For $\scr F_s\in\{\scr F_t\}_{t\in[0,1]}$, we define the expectation of the random variable $\xi$ under $\scr F_s$ by 
	$$\bb E_s[\xi]=\bb E[\xi|\scr F_s].$$
	Without ambiguity, we write $\bb E_s\xi$ for short. For expectations (resp. conditional expectations) on $(\Omega_L, \scr F^L,\bb P_L)$ and $(\Omega_{\tau}, \scr F^{\tau},
	\bb P_{\tau})$, we will use $\bb E^L$ (resp. $\bb E_s^L$) and $\bb E^\tau$ (resp. $\bb E^\tau_s$) respectively. 
	
	\subsection{ The L\' evy process and assumptions}
    	Recall that $(L_t)_{t\ge0}$ is a L\' evy process on  $(\Omega_L, \scr F^L,\bb P_L)$ (i.e., it is continuous in probability, it has stationary increments, c\`adl\` ag trajectories, $L_t -L_s$ is independent of $\scr F_s$, $0 \le s \le t$, and $L_0 = 0$) with the additional property that its characteristic function is given by
	$$
	\bb E^L\rrm e^{i\left\langle L_t,u \right\rangle } = \rrm e^{-t\Psi(u)},\quad \Psi(u) =-\int_{\bb R^d}\big(\rrm e^{i\left\langle u,z\right\rangle }-1-i\left\langle u,z\right\rangle \I_{\{|z|\le1\}}(z)\big)\upsilon(\rrm dy),
	$$
	where $u \in\bb R^d$, $t \ge0$ and $\upsilon$ is a measure such that
	$$\upsilon(D)=\int_{\bb S}\mu(\rrm d\xi)\int_{0}^{\infty}\I_D(r\xi)\frac{\rrm dr}{r^{1+\alpha}},\quad D\in\scr B(\bb R^d),$$
	$\bb S=\{x\in\bb R^d,|x|=1\}$ and $\mu$ is a finite symmetric measure on $(\bb S, \scr B(\bb S))$, i.e. $\mu(A) = \mu(-A)$, for any $A \in\scr B(\bb S)$. In this paper, we  consider $(L_t)_{t\ge0}$ being $d$-dimensional symmetric $\alpha$-stable process so that $\Psi$ has the following representation:
	$$\Psi(u)=\int_{\bb R^d}[1-\cos\left\langle u,x\right\rangle ]\upsilon(\rrm dx).$$ 
It is easy to see that the L\'evy process $L_t$ has the following two properties.
	\begin{itemize}
		\item Scaling property: for $t> 0$, let $\mu_t$ denotes the law of $L_t$ , $t>0$, then, $\mu_t (A) = \mu_1 (t^{1/\alpha} A)$, for any $A \in\scr B(\bb R^d )$, $t > 0$.
		\item  For any $\gamma>\alpha$, we have, $$\int_{|z|\le 1}|z|^\gamma\upsilon(\rrm dz)<\infty.$$
	\end{itemize}
	By the L\'evy-It\^o decomposition, we can write that 
$$L_t=\int_{0}^{t}\int_{|z|\le1}z\tilde N(\rrm ds,\rrm dz)+\int_{0}^{t}\int_{|z|>1}z N(\rrm ds,\rrm dz),\quad t\ge0,$$
	where $N$ is the Poisson random measure defined as follows:
	$$N([0,t],U)=\sum_{0\le s\le t}\I_{U}(L_s-L_{s^-}),\qquad U\in\scr B(\bb R^d_0),\ t\ge0, $$ and $\tilde N$ defined by  $$\tilde N(\rrm ds,\rrm dz):=N(\rrm ds,\rrm dz)-\upsilon(\rrm dz)\rrm ds. $$
We will use $P_t$ to denote the Markov transition semigroup of the L\'evy process $L$ with  infinitesimal generator $\al L$, i.e. $$P_tf(x)=\bb E^Lf(x+L_t), \quad f\in C^0_b(\bb R^d), $$
	and the operator $\al L$ can be rewritten as 
	$$\al Lf(x)=\int_{\bb R^d}\big(f(x+z)-f(x)-\I_{\{|z|\le1\}}\left\langle z,\nabla f(x) \right\rangle \big)\upsilon(\rrm dz),\quad f\in C^\infty_c(\bb R^d).$$
	Because $P$ is the Markovian semigroup of symmetric $\alpha$-stable process, for any measurable function $f\in C_b^{\beta,\eta}$ and any $\scr F^L_s$-measurable random vector $\xi$, then 
	$$\bb E^L_s\big[f(r,L_t+\xi)\big]=\bb E^L_s\big[f(r,(L_t-L_s)+L_s+\xi)\big]=\bb E^L\big[f(r,(L_t-L_s)+x+\xi)\big]|_{x=L_s}=P_{t-s}f(r,L_s+\xi) \quad \text{for any } r\in[0,1].$$ 
	Following \cite{butkovsky2024strongrateconvergenceeuler,chenandzhanglevy}, we impose the following assumptions on $P$ and $\al L$.
	\begin{assumption}\label{H1}\cite[Assumption H1]{butkovsky2024strongrateconvergenceeuler}
		There exists a constant $M$ such that for any $f\in C^0_b(\bb R^d)$, one has 
		$$\left\|\nabla P_tf \right\|_{C^0_b}\le Mt^{-1/\alpha}\left\|f \right\|_{C^0_b},\quad t\in(0,1]. $$
	\end{assumption}
	\begin{assumption}\label{H2}\cite[Assumption H2]{butkovsky2024strongrateconvergenceeuler}
		For $\delta\in \{0,1\}$ and any $\varepsilon>0$, there exists a
		constant $M$ such that for any $f\in C_b^{\alpha+\delta+\varepsilon}(\bb R^d)$ vanishing at infinity, one has 
		$$
			\left\|\al Lf \right\|_{C^\delta_b}\le Mt^{-1/\alpha}\left\|f \right\|_{C_b^{\delta+\alpha+\varepsilon}},\quad t\in(0,1]. $$
	\end{assumption}
	\begin{assumption}\label{H3}\cite[Assumption H3]{butkovsky2024strongrateconvergenceeuler}
		For any $p\in(0,\alpha)$, $\varepsilon>0$, there exists a constant $M$ such that
		$$\bb E^L\big[|L_t|^p\wedge1\big]\le Mt^{p/\alpha-\varepsilon},\quad t\in(0,1].$$
	\end{assumption}
	
	Also, we know the semigroup has these properties.
	
	\begin{proposition}\cite[Proposition 3.3]{butkovsky2024strongrateconvergenceeuler}
		If $L$ is a $d$-dimensional symmetric $\alpha$-stable process and Assumptions \ref{H1}, \ref{H2} hold, for every $\varepsilon>0,\  \beta,\eta\in(0,1],\ \rho$ as below,  and $\mu\in(\varepsilon\vee (\eta-\rho)/\alpha,1+\varepsilon]$, then there exists $C$ depending on $\alpha,\varepsilon,\eta,\rho,\mu$ such that for any $f\in C^{\beta,\eta}_b$:
		\begin{align}
			&\left\|P_tf(r,\cdot) \right\|_{C_b^\rho}\le C\left\| f\right\| _{C_b^{\beta,\eta}}t^{\frac{(\eta-\rho)\wedge0}{\alpha}} \qquad 0\le t, r\le 1,\ \rho\ge0,\\
			&\left\|P_tf(r,\cdot)-P_sf(r,\cdot) \right\|_{C_b^\rho}\le C\left\| f\right\| _{C_b^{\beta,\eta}}(t-s)^{\mu-\varepsilon}s^{\frac{(\eta-\rho)}{\alpha}-\mu} \quad 0\le s<t\le1 ,\ 0\le r\le1,\ \rho\in\{0,1\}.
		\end{align}
	\end{proposition}
	
	\subsection{The main result}
	Define the simplex $\Delta_{S,T}:=\{(s,t)|S\le s\le t\le T, s-(t-s)\ge S\}$. For a function $f:\Delta_{S,T}\to\bb R^d$, and a triplet of time $(s,u,t)$ such that $S\le s\le u\le t\le T$, we define $$\delta f_{s,u,t}:=f_{s,t}-f_{s,u}-f_{u,t}.$$ 
	Next we give the stochastic sewing lemma.
	\begin{theorem}[Stochastic sewing lemma {\cite[Theorem 2.4]{Khoa2020}}]\label{classical sew}
		Consider a probability space $(\Omega, \mathcal{F},\{\mathcal{F}_t\}_{t\geq 0},\bb P)$. Let $p\in[2,\infty)$, $0\leq S\leq T\le1$ and let $(A_{s,t})_{(s,t)\in\Delta_{S,T}}$ be a family of random variables in $ L^p(\Omega;\bb R^d)$ such that $A_{s,t}$ is $\scr F_t$-measurable. Suppose that for some $\varepsilon_1,\varepsilon_2>0$ and $C_1,C_2$ the bounds
		\begin{align}
			& \|A_{s,t}\|_{L^p(\Omega;\bb R^d)}  \leq C_1|t-s|^{1/2+\varepsilon_1},\label{SSL1}
			\\
			&\|\bb E[\delta A_{s,u,t}|\scr F_s]\|_{L^p(\Omega;\bb R^d)}=:\|\bb E^s[\delta A_{s,u,t}]\|_{L^p(\Omega;\bb R^d)}  \leq C_2 |t-s|^{1+\varepsilon_2},\label{SSL2}
		\end{align}
		hold for all $S\leq s\leq u\leq t\leq T$.
		
		Then there exists a unique (up to modification) $\{\scr F_t\}$-adapted 
		process $\al A:[0,1]\to L^p(\Omega;\bb R^d)$ such that $\al A_0=0$ and the following bounds hold for some positive constants $K_1,K_2$. For any $S\leq s\leq u\leq t\leq T$, one has:
		\begin{align}
			&\|\al A_t	-\al A_s-A_{s,t}\|_{L^p(\Omega;\bb R^d)}  \leq K_1 |t-s|^{1/2+\varepsilon_1}+K_2 |t-s|^{1+\varepsilon_2},\label{SSL1 cA}
			\\
			&\|\bb E_s\big[\al A_t	-\al A_s-A_{s,t}\big]\|_{L^p(\Omega;\bb R^d)}  \leq K_2|t-s|^{1+\varepsilon_2},\label{SSL2 cA}.
		\end{align}
		Moreover, there exists a positive constant $K$ depending only on $\varepsilon_1,\varepsilon_2$ and $d$ such that for any  $S\leq s\leq u\leq t\leq T$, $\al A$ satisfies the bound:
		\begin{equation}\label{SSL3 cA}
			\|\al A_t-\al A_s\|_{L^p(\Omega;\bb R^d)}  \leq  KpC_1 |t-s|^{1/2+\varepsilon_1}+KpC_2 |t-s|^{1+\varepsilon_2},\quad (s,t)\in[0,1].
		\end{equation}
	\end{theorem}
	However, a symmetric $\alpha$-stable process $(L_t)_{t \ge 0}$ admits finite moments only up to order $k \le \alpha$. Since our goal is to estimate the convergence rate between $X_t$ and its approximation $X^{(n)}_t$ in the norm $\| \cdot \|_{L^p(\Omega)}$ for some $p \ge \alpha$, the standard moment-based techniques are insufficient. To overcome this challenge, we will also use a conditional shifted version of the stochastic sewing lemma\textemdash  a refinement of the classical SSL adapted to our framework.

	\begin{theorem}(Conditional shifted version of SSL {\cite[Lemma 3.1]{butkovsky2024strongrateconvergenceeuler}})	\label{condition sew}
		Let $0\le S\le T\le1$, $p\in[2,\infty)$ and $(A_{s,t})_{(s,t)\in\Delta_{S,T}}$ be a family of random variables in $L^p(\Omega,\bb R^d)$ such that $A_{s,t}$ is $\scr F_t$-measurable. Let $\scr G\subset \scr F_S$ is a $\sigma$-algebra. There exist $\varepsilon_1, \varepsilon_2, \varepsilon_3>0$ and $\scr G$ random variables $\Gamma_1, \Gamma_2, \Gamma_3\ge0$ such that 
		\begin{align}
			&\big\|A_{s,t} \big\|_{L^p(\Omega)|\scr G}\le \Gamma_1\left|t-s \right|^{1/2+\varepsilon_1},  \\
			&\big\|\bb E_{s-(t-s)}\delta A_{s,u,t} \big\| _{L^p(\Omega)|\scr G}\le \Gamma_2\left|t-s \right| ^{1+\varepsilon_2}+\Gamma_3\left|t-s \right| ^{1+\varepsilon_3},
		\end{align}
		for all $(s,t)\in\Delta_{S,T}$ and $u=(s+t)/2$. 
		
		Further, there is a process $\al A=\{\al A_t, t\in[S,T]\}$ such that for any $S\le s\le t\le T$, it yields that $$\al A_t-\al A_s=\lim\limits_{m\to\infty}\sum_{i=1}^{m-1}A_{s+i\frac{t-s}{m},s+(i+1)\frac{t-s}{m}}\quad\text{in probability.}$$
		Then there exist deterministic constants $K_1,K_2,K_3>0$ only depending on $\varepsilon_1,\varepsilon_2,\varepsilon_3,p$ and $d$ such that for any $S\le s\le t\le T$, we obtain that $$\left\|\al A_t-\al A_s \right\|_{L^p(\Omega)|\scr G}\le K_1\Gamma_1\left|t-s \right|^{1/2+\varepsilon_1}+K_2\Gamma_2\left|t-s \right| ^{1+\varepsilon_2}+K_3\Gamma_3\left|t-s \right| ^{1+\varepsilon_3} .$$
	\end{theorem}
		In order to transform \eqref{main SDE} to an SDEs with regular coefficients, for any $\lambda> 0$, consider the following resolvent equation on $\bb R^d$:
	
	\begin{equation}\label{kolmo}
		\partial_tu-\lambda u+\al Lu+b\nabla u+b=0.
	\end{equation}
	
	\begin{theorem}\cite[Corollary 2.10]{chenandzhanglevy}\label{backward}
		Suppose $\alpha\in(1,2)$ and Assumption \ref{H1} holds. If $\left\|b\right\| _{C_b^{\beta,\eta}}<\infty$ holds for $\beta,\eta\in(0,1]$, then the solution $u:[0,1]\times\bb R^d\to\bb R^d$ of equation \eqref{kolmo} such that 
		$$ u(t,x)=\int_{t}^{1} (\al L-\lambda)u(s,x)\rrm ds+\int_{t}^{1} b(s,x)\nabla u(s,x)\rrm ds +\int_{t}^{1} b(s,x)\rrm ds,$$
		with $$\sup_{0\le t\le 1}\|u(t,\cdot)\|_\infty\le\sup_{0\le t\le 1}\|b(t,\cdot)\|_\infty, $$
		and for some $\theta_0>0$ and any $\lambda\ge0$,
        \begin{equation}\label{eq,unorm}
        \sup_{0\le t\le 1}\left\|\nabla u(t,\cdot) \right\|_{C_b^{\varpi}}\le C(1\vee\lambda)^{-\theta_0}\left\|b \right\|_{C_b^{\beta,\eta}},
        \end{equation}
		where $\|f\|_\infty=\sup_{x\in\bb R^d}|f(x)|$ and $\varpi=(\eta/2+\alpha-1)\wedge 1$.
	\end{theorem}
	
	Finally, we give the main result in this paper.
	\begin{theorem}\label{theo2.3}
		Let $L$ be a d-dimensional symmetric $\alpha$-stable process, $\alpha\in(1,2)$. Assume Assumptions \ref{H1}-\ref{H3}. Take SDE \eqref{main SDE} with the drift coefficient function $b\in C_b^{\beta,\eta}$ such that $\beta,\eta\in(0,1]$ and
		\begin{equation}\label{co:AS}
			2\eta+\alpha>2\text{ and } (\beta+1)\alpha+\eta>2.
		\end{equation}
		Consider the solution $X_t$ of \eqref{main SDE} and its numerical approximation $X^{(n)}_t$ by the randomised Euler-Maruyama scheme \eqref{con EM SDE}. Then for any $p\ge1$ and $\varepsilon\in(0,1/2)$, there exists a positive constant $C>0$ such that 
		$$\bb E\Big[\sup_{0\le t\le 1}\big|X_t-X^{(n)}_t\big|^p\Big]\le Cn^{-(1/2+\gamma-\varepsilon)p},$$
		where $\gamma:=\beta\wedge(\eta/\alpha)\wedge(1/2)$.
	\end{theorem}
	\begin{remark}
    \begin{itemize}
        \item The case of $\alpha=2$, though not being covered in Theorem \ref{theo2.3}, has been considered in \cite{bao2025randomisedeulermaruyamamethodsdes}.
        \item It is well-known that for a symmetric $\alpha$-stable process, and therefore $X_t$ (resp. $X_t^{(n)}$), moments exist only up to order $\alpha$, and not necessarily for any $p \ge 1$. However, 
        given that $b\in C_b^{\beta,\eta}$ is bounded and SDE \eqref{main SDE} is driven by additive noise, we have the pointwise bound: $$ b(t,X_t)\le \|b\|_{C_b^{\beta,\eta}} \text{ \big(resp. } b(t,X_t^{(n)})\le \|b\|_{C_b^{\beta,\eta}}\text{\big)}. $$
        As a consequence, the difference $X_t - X_t^{(n)}=\int_0^t \big(b(t,X_t)- b(t,X_t^{(n)})\big)\mathrm{d}s$ admits moments of all orders $p \ge 1$.
        \item 
        In comparison with \cite[Theorem 2.2]{butkovsky2024strongrateconvergenceeuler}, we consider the case where the drift term $b$ is time-dependent. Specifically, we assume that $b(t,x)$ is $\beta$-H\"older continuous with respect to time. Under this setting, the order of convergence of EM method considered in \cite{butkovsky2024strongrateconvergenceeuler} cannot exceed $$\beta\wedge\Big(\frac{1}{2}+\big(\frac{\eta}{\alpha}\wedge\frac{1}{2}\big)-\varepsilon\Big),$$ for any $\varepsilon>0$, which yields a weaker result compared to ours. 
    \end{itemize}
	    
	\end{remark}
    
	\section{Priori Estimations}\label{sec:priori}
	Recall the process $\big(X_t^{(n)}\big)_{t \in [0,1]}$ defined in \eqref{con EM SDE}. We now introduce $\phi_t^{(n)}$ to isolate the drift component of the Euler-Maruyama scheme, defined by

	\begin{equation}\label{eqiso}
		\begin{aligned}
			\phi_t^{(n)}&=X_t^{(n)}-L_t\\&=x_0+\int_{0}^{t}b(\kappa_n^\tau(s),X_{\kappa_n(s)}^{(n)})\rrm ds\\&=
			x_0+\int_{0}^{t}b(\kappa_n^\tau(s),\phi_{\kappa_n(s)}^{(n)}+L_{\kappa_n(s)})\rrm ds
		\end{aligned}
	\end{equation}
	It is straightforward to verify that $\phi_t^{(n)}$ is $\scr F_{\kappa_n(t)}$-adapted.
	\begin{lemma}\label{X_t-x_s}
		Consider the randomised EM simulation \eqref{con EM SDE}. Suppose that the function $b$ is bounded. If $L$ is a $d$-dimensional symmetric $\alpha$-stable process with $\alpha\in(1,2)$, $(X_t^{(n)})_{t \in [0,1]}$ in \eqref{con EM SDE}, then there exists a constant $C>0$ and for any $0\le s\le t\le 1$ such that $$\big\|X^{(n)}_t-X^{(n)}_{s} \big\|_{L^1(\Omega)}\le C  |t-s|^{\frac{1}{\alpha}-\varepsilon}. $$ 
		Specially, for any $p\ge 1$, we have 
		$$\big\|\phi^{(n)}_t-\phi^{(n)}_{s} \big\|_{L^p(\Omega)}\le C  |t-s|. $$ 
	\end{lemma}
	\begin{proof}
		Recall from \eqref{eqiso} that for any $t\in[0,1]$, we have 
		\begin{equation*}
			\begin{aligned}
				\left|X^{(n)}_t-X^{(n)}_{s}\right| &\le	\left|\phi^{(n)}_t-\phi^{(n)}_{s}\right|+\left|L_t-L_s\right| \\&\le \int_{s}^{t}\left|b(\kappa^\tau_n(r),\phi^{(n)}_{\kappa_n(r)}+L_{\kappa_n(r)}) \right|\rrm dr +\left|L_t-L_{s }\right|\\&\le
				\left\|b \right\|_{1,\infty} |t-s| +\left|L_t-L_{s}\right|,
			\end{aligned}
		\end{equation*}
        where $\|b\|_{1,\infty}=\sup_{t\in[0,1],x\in\bb R^d}|b(t,x)|$. By using the scaling property, we have $\bb E|L_t-L_s|=(t-s)^{1/\alpha}\bb E|L_1|.$
		Therefore, we obtain 
		\begin{align*}
			\big\|X^{(n)}_t-X^{(n)}_{s} \big\|_{L^1(\Omega)}\le C|t-s|^{\frac{1}{\alpha}}.
		\end{align*}
	\end{proof}
	
	\begin{lemma}\label{lem3.1}
		Assume $g_1\in C_b^{\beta,\eta}$ and $g_2\in C([0,1];C^0_b(\bb R^d))$, and for any $0\le S\le T\le1$, $\varepsilon>0$ and $p\ge2$, then there exists a constant $C>0$ such that
		\begin{align}
			&\left\|\int_S^T \big(g_1(r,X^{(n)}_{\kappa_n(r)})-g_1(\kappa_n^\tau(r),X^{(n)}_{\kappa_n(r)})\big) \rrm dr\right\|_{L^p(\Omega)}\le C|T-S|^{1/2+\varepsilon}n^{-(1/2+\beta-\varepsilon)}\label{eq9};\\
			&\left\|\int_S^T \big(g_1(r,X^{(n)}_{\kappa_n(r)})-g_1(\kappa_n^\tau(r),X^{(n)}_{\kappa_n(r)})\big)g_2(r,X_{r}^{(n)}) \rrm dr\right\|_{L^p(\Omega)}\le C|T-S|^{1/2+\varepsilon}n^{-(1/2+\beta-\varepsilon)} .\label{eq10}
		\end{align}
	\end{lemma}
	\begin{proof}
		Firstly, we prove equation \eqref{eq9}. Define for $0\le S\le T\le 1$ $$A_{s,t}:=\bb E_s\Big[\int_{s}^{t}(g_1(r,X^{(n)}_{\kappa_n(r)})-g_1(\kappa_n^\tau(r),X^{(n)}_{\kappa_n(r)}))\rrm dr\Big],\quad (s,t)\in\Delta_{S,T}.$$
		For any $ S\le s\le u\le t\le T$, we have
		\begin{equation*}
			\begin{aligned}
				\delta A_{s,u,t}&=A_{s,t}-A_{s,u}-A_{u,t}\\
				&=\bb E_s\big[\int_{u}^{t}(g_1(r,X^{(n)}_{\kappa_n(r)})-g_1(\kappa_n^\tau(r),X^{(n)}_{\kappa_n(r)}))\rrm dr\big]\\&\quad -\bb E_u\big[\int_{u}^{t}(g_1(r,X^{(n)}_{\kappa_n(r)})-g_1(\kappa_n^\tau(r),X^{(n)}_{\kappa_n(r)}))\rrm dr\big].
			\end{aligned}
		\end{equation*}
		It is easy that 
		$$
		\bb E_s[\delta A_{s,u,t}]=0.
		$$
		When $t-s<1/n$, we know 
		\begin{equation*}
			\begin{aligned}
				\big|g_1(r,X^{(n)}_{\kappa_n(r)})-g_1(\kappa_n^\tau(r),X^{(n)}_{\kappa_n(r)}) \big|  \le \left\| g_1\right\|_{C_b^{\beta,\eta}} \left|r-\kappa_n^\tau(r) \right|^\beta .
			\end{aligned}
		\end{equation*}
		Then it is easy to get
		$$\left\|\int_s^t \big(g_1(r,X^{(n)}_{\kappa_n(r)})-g_1(\kappa_n^\tau(r),X^{(n)}_{\kappa_n(r)})\big) \rrm dr\right\|_{L^p(\Omega)}\le \left\| g_1\right\| _{C_b^{\beta,\eta}}|t-s|^{1/2+\varepsilon}n^{-(1/2+\beta-\varepsilon)} .$$
		For another, if $t-s\ge 1/n$, then there exist indices $i,j\in\{1,2,\cdots,n\}$ such that $i/n\le s<(i+1)/n$ and $j/n\le t<(j+1)/n$. Thus, we have
		\begin{equation*}
			\begin{aligned}
				|A_{s,t}|&=\Big|\int_{s}^{t}\bb E_s\big[g_1(r,X^{(n)}_{\kappa_n(r)})-g_1(\kappa_n^\tau(r),X^{(n)}_{\kappa_n(r)})\big] \rrm dr\Big|\\&\le \int_{s}^{\frac{i+1}{n}}\Big|\bb E_s\big[g_1(r,X^{(n)}_{\kappa_n(r)})-g_1(\kappa_n^\tau(r),X^{(n)}_{\kappa_n(r)})\big]\Big| \rrm dr\\&\quad+\sum_{k=i+1}^{j-1}\Big|\int_{\frac{k}{n}}^{\frac{k+1}{n}}\bb E_s\big[g_1(r,X^{(n)}_{\kappa_n(r)})-g_1(\kappa_n^\tau(r),X^{(n)}_{\kappa_n(r)})\big] \rrm dr\Big|\\&\quad+\int_{\frac{j}{n}}^{t}\Big|\bb E_s\big[g_1(r,X^{(n)}_{\kappa_n(r)})-g_1(\kappa_n^\tau(r),X^{(n)}_{\kappa_n(r)})\big]\Big| \rrm dr\\&=:I_s+\sum_{k=i+1}^{j-1}I_k+I_t.
			\end{aligned}
		\end{equation*}
		Next, we just estimate $I_k$ by conditional expectation.
		\begin{equation*}
			\begin{aligned}
				&\int_{\frac{k}{n}}^{\frac{k+1}{n}}\bb E_s\Big[\bb E_{k/n}\big[g_1(r,X^{(n)}_{\kappa_n(r)})-g_1(\kappa_n^\tau(r),X^{(n)}_{\kappa_n(r)})\big]\Big] \rrm dr\\&=\int_{\frac{k}{n}}^{\frac{k+1}{n}}\bb E_sg_1(r,X^{(n)}_{k/n})\rrm dr-\frac{1}{n}\bb E_s\Big[\bb E^L_{k/n}\big[\bb E^\tau_{k/n} g_1((k+\tau)/n,X^{(n)}_{k/n})\big]\Big] \\&=\int_{\frac{k}{n}}^{\frac{k+1}{n}}\bb E_sg_1(r,X^{(n)}_{k/n})\rrm dr-\frac{1}{n}\int_{0}^{1}\bb E_s g_1((k+h)/n,X^{(n)}_{k/n})\rrm dh\\&=0	
			\end{aligned}
		\end{equation*}
		Therefore, under this condition, we obtain
		\begin{equation*}
			\begin{aligned}
				\left\|\int_s^t (g_1(r,X^{(n)}_{\kappa_n(r)})-g_1(\kappa_n^\tau(r),X^{(n)}_{\kappa_n(r)})) \rrm dr\right\|_{L^p(\Omega)}&\le \left\|I_s\right\|_{L^p(\Omega)}+\left\|I_t\right\|_{L^p(\Omega)}\\&\le  2\left\|g_1 \right\|_{C_b^{\beta,\eta}} |t-s|^{1/2+\varepsilon}n^{-(1/2+\beta-\varepsilon)}
			\end{aligned}
		\end{equation*}
		Then, we prove equation \eqref{eq10}. Also using the same methods,
			\begin{equation*}
			\begin{aligned}
				\Big\|\int_{s}^{t}\big(g_1(r,X^{(n)}_{\kappa_n(r)})-g_1(\kappa_n^\tau(r),X^{(n)}_{\kappa_n(r)})\big)g_2(\kappa_n(r),X^{(n)}_{\kappa_n(r)}) \rrm dr\Big\|_{L^p(\Omega)}\le 2\left\|g_1 \right\|_{C_b^{\beta,\eta}}\left\|g_2 \right\|_{[s,t]} |t-s|^{1/2+\varepsilon}n^{-(1/2+\beta-\varepsilon)}.
			\end{aligned}
		\end{equation*}
		In order to prove equation \eqref{eq10}, define the process as follows
		$$\al A_t:=\int_{0}^{t}\big(g_1(r,X^{(n)}_{\kappa_n(r)})-g_1(\kappa_n^\tau(r),X^{(n)}_{\kappa_n(r)})\big)g_2(r,X_{r}^{(n)}) \rrm dr,$$
		and define a process in $\Delta_{S,T}$,
		$$A_{s,t}:=\int_{s}^{t}\big(g_1(r,X^{(n)}_{\kappa_n(r)})-g_1(\kappa_n^\tau(r),X^{(n)}_{\kappa_n(r)})\big)g_2(\kappa_n(r),X^{(n)}_{\kappa_n(r)}) \rrm dr, \quad (s,t)\in \Delta_{S,T}.$$
		For any $S\le s\le u\le t\le T$, we have $\bb E_s[\delta A_{s,u,t}]=0$. Then, it can be easily verified as follows
		\begin{equation*}
			\begin{aligned}
				&\left\|\al A_t-\al A_s-A_{s,t} \right\|_{L^p(\Omega)} \\&\le \left\|\int_{s}^{t}\big(g_1(r,X^{(n)}_{\kappa_n(r)})-g_1(\kappa_n^\tau(r),X^{(n)}_{\kappa_n(r)})\big)\big(g_2(r,X_{r}^{(n)})-g_2(\kappa_n(r),X^{(n)}_{\kappa_n(r)})\big) \rrm dr\right\|_{L^p(\Omega)}\\&\le C\left\|g_1 \right\|_{C_b^{\beta,\eta}}\left\|g_2 \right\|_{[s,t]}|t-s|^{1+\beta}.
			\end{aligned}
		\end{equation*}
		By Sewing Lemma (Theorem \ref{classical sew}), we prove this conclusion.
	\end{proof}
	Next, we derive estimates for the family $(\phi^{(n)}_t)_{t\ge0}$ as defined in \eqref{eqiso}.

	\begin{lemma}\label{lam3.2}
		Suppose the function $b\in C^{\beta,\eta}_b$ and $L$ is a $d$-dimensional symmetric $\alpha$-stable process with $\alpha\in(1,2)$. For any $0\le s\le t\le1$, $p\ge1$ and $\varepsilon>0$, there exists a constant $C>0$ such that
		$$\big\|\phi^{(n)}_t-\bb E_s\phi^{(n)}_t \big\|_{L^p(\Omega)|\scr F_s}\le C\left| t-s\right|^{1+\beta\wedge\frac{\eta}{\alpha}\wedge\frac{1}{p}-\varepsilon}. $$
	\end{lemma}
	\begin{proof}
		By conditional Minkovski's inequality, Lemma \ref{X_t-x_s} and using \eqref{gprop}, we have that
		\begin{equation*}
			\begin{aligned}
				&\big\|\phi^{(n)}_t-\bb E_s\phi^{(n)}_t \big\|_{L^p(\Omega)|\scr F_s}\\ &=\Big\|\phi^{(n)}_t-\phi^{(n)}_s-\bb E_s\Big[\int_{s}^{t}b(\kappa^\tau_n(r),\phi^{(n)}_{\kappa_n(r)}+L_{r})\rrm dr\Big] \Big\|_{L^p(\Omega)|\scr F_s}\\&\le \Big\|\phi^{(n)}_t-\phi^{(n)}_s-\int_{s}^{t}b(\kappa^\tau_n(s),\phi^{(n)}_{\kappa_n(s)}+L_{s})\rrm dr\Big\|_{L^p(\Omega)|\scr F_s}\\&\quad+\Big\|\bb E_s\Big[\int_{s}^{t}b(\kappa^\tau_n(r),\phi^{(n)}_{\kappa_n(r)}+L_{r})\rrm dr\Big]-\int_{s}^{t}b(\kappa^\tau_n(s),\phi^{(n)}_{\kappa_n(s)}+L_{s})\rrm dr \Big\|_{L^p(\Omega)|\scr F_s}
				\\&\le  2\int_{s}^{t}\left\|\big(b(\kappa_n^\tau(r),\phi^{(n)}_{\kappa_n(r)}+L_{\kappa_n(r)})-b(\kappa^\tau_n(s),\phi^{(n)}_{\kappa_n(s)}+L_{s})\big)\right\|_{L^p(\Omega)}\rrm dr \\&\le
				C\int_{s}^{t}\big(|r-s|^{\beta}+\left\|\big(|L_{\kappa_n}(r)-L_s|^\eta +|\phi^{(n)}_{\kappa_n(r)}-\phi^{(n)}_{\kappa_n(s)}|^\eta\big)\wedge1\right\|_{L^p(\Omega)} \big)\rrm dr\\&\le
				C\left|t-s \right|^{1+\beta\wedge\frac{\eta}{\alpha}\wedge\frac{1}{p}-\varepsilon} 
			\end{aligned}
		\end{equation*}
		The second inequality holds because the integral $\int_{s}^{t}b(\kappa_n^\tau(s),\phi^{(n)}_{\kappa_n(s)}+L_s)\rrm dr$ is $\scr F_s$-measurable.
	\end{proof}
	\begin{lemma}\label{lem3.3}
		Suppose $b\in C^{\beta,\eta}_b$ and $L$ is a $d$-dimensional symmetric $\alpha$-stable process with $\alpha\in(1,2)$. For any $0\le s\le t\le1$ and $\varepsilon>0$ there exists a constant $C>0$ such that 
		$$ \big\|\phi^{(n)}_t-\phi^{(n)}_{\kappa_n(t)}-\bb E_s\big[ \phi^{(n)}_t-\phi^{(n)}_{\kappa_n(t)} \big] \big\|_{L^1(\Omega)|\scr F_s}\le Cn^{-\frac{1}{2}-(\beta\wedge\frac{\eta}{\alpha})+\varepsilon}\left| t-s\right| ^{\frac{1}{2}+\varepsilon}. $$
	\end{lemma}
	\begin{proof}
		Because 
		\begin{equation*}
			\begin{aligned}
				\phi^{(n)}_t-\phi^{(n)}_{\kappa_n(t)}=\int_{\kappa_n(t)}^{t}b(\kappa_n^\tau(t),L_{\kappa_n(t)}+\phi^{(n)}_{\kappa_n(t)})\rrm ds.
			\end{aligned}
		\end{equation*}
		If $\kappa_n(t)\le s$, because $\phi^{(n)}_{\kappa_n(t)}$ and $\phi^{(n)}_{t}$ are both $\scr F_{\kappa_n(t)}$-measurable so that $\scr F_s$-measurable, then we have
		 $\phi^{(n)}_t-\phi^{(n)}_{\kappa_n(t)}-\bb E_s\big[ \phi^{(n)}_t-\phi^{(n)}_{\kappa_n(t)} \big]=0$. On the other hand, if $\kappa_n(t)> s$, based on \eqref{gprop} we can conclude that 
		\begin{equation*}
			\begin{aligned}
				&\bb E_s\Big|\phi^{(n)}_t-\phi^{(n)}_{\kappa_n(t)}-\bb E_s\big( \phi^{(n)}_t-\phi^{(n)}_{\kappa_n(t)} \big)\Big|\\&=\bb E_s\Big|\int_{\kappa_n(t)}^{t}b(\kappa_n^\tau(r),L_{\kappa_n(r)}+\phi^{(n)}_{\kappa_n(r)})\rrm dr-\bb E_s\Big[\int_{\kappa_n(t)}^{t}b(\kappa_n^\tau(r),L_{\kappa_n(r)}+\phi^{(n)}_{\kappa_n(r)})\rrm dr\Big]\Big|\\&\le \left|t-\kappa_n(t) \right| \bb E_s\big|b(\kappa_n^\tau(t),L_{\kappa_n(t)}+\phi^{(n)}_{\kappa_n(t)})-\bb E_sb(\kappa_n^\tau(t),L_{\kappa_n(t)}+\phi^{(n)}_{\kappa_n(t)}) \big|\\&
				\le \left|t-\kappa_n(t) \right|\bb E_s\big|b(\kappa_n^\tau(t),L_{\kappa_n(t)}+\phi^{(n)}_{\kappa_n(t)})-\bb E_s\big[\bb E^\tau_{\kappa_n(t)}b(\kappa_n^\tau(t),L_{\kappa_n(t)}+\phi^{(n)}_{\kappa_n(t)})\big] \big|\\&\le\left|t-\kappa_n(t) \right|\int_{0}^{1}\bb E_s\big|b(\kappa_n^\tau(t),L_{\kappa_n(t)}+\phi^{(n)}_{\kappa_n(t)})-\bb E_s\big(b(\kappa_n(t)+h/n,L_{\kappa_n(t)}+\phi^{(n)}_{\kappa_n(t)})\big)  \big|\rrm dh\\&\le
				\left|t-\kappa_n(t) \right|\left( \int_{0}^{1}\bb E_s\big|b(\kappa_n^\tau(t),L_{\kappa_n(t)}+\phi^{(n)}_{\kappa_n(t)})-b(\kappa_n(t),L_{s}+\phi^{(n)}_{s})\big|\rrm dh\right. \\&\qquad\qquad\qquad\quad+\left.\int_{0}^{1}\bb E_s\big|b(\kappa_n(t)+h/n,L_{\kappa_n(t)}+\phi^{(n)}_{\kappa_n(t)})-b(\kappa_n(t),L_{s}+\phi^{(n)}_{s})\big|\rrm dh \right) \\&\le
				C\left|t-\kappa_n(t) \right|\int_{0}^{1}\bb E_s\Big[|\kappa_n^\tau(t)-\kappa_n(t)|^\beta+|h/n|^\beta+\big(\big|L_{\kappa_n(t)}-L_{s}\big|^\eta\wedge1\big)+\big|\phi^{(n)}_{\kappa_n(t)}-\phi^{(n)}_s\big|^\eta\Big]\rrm dh.
			\end{aligned}
		\end{equation*}	
		Therefore, applying Minkowski's inequality and Jensen's inequality, and using the bound $\left| t-\kappa_n(t)\right| \le n^{-1}\wedge\left| t-s\right| $, we obtain
		\begin{equation*}
			\begin{aligned}
				&\big\|\phi^{(n)}_t-\phi^{(n)}_{\kappa_n(t)}-\bb E_s\big[ \phi^{(n)}_t-\phi^{(n)}_{\kappa_n(t)} \big] \big\|_{L^1(\Omega)|\scr F_s}\\&\le 
				C\left|t-\kappa_n(t) \right|\Big(n^{-\beta}+\big\|\big|L_{\kappa_n(t)}-L_{s}\big|^\eta\wedge1\big\|_{L^1(\Omega)|\scr F_s}+\big\|\phi^{(n)}_{\kappa_n(t)}-\phi^{(n)}_s\big\|_{L^1(\Omega)|\scr F_s}^\eta\Big)\\&\le
				Cn^{-\frac{1}{2}-(\beta\wedge\frac{\eta}{\alpha})+\varepsilon}\left| t-s\right| ^{\frac{1}{2}+\varepsilon}.
			\end{aligned}
		\end{equation*}
	\end{proof}
	
	\begin{proposition} \label{prop3.4}
		Let $\beta,\eta\in(0,1]$ and take $\varepsilon\in (0,1/2)$, and $L$ is a $d$-dimensional symmetric $\alpha$-stable process with $\alpha\in(1,2)$ and suppose that \eqref{co:AS} holds. Then for any $g\in C_b^{\beta,\eta}$, there exists a constant $C>0$ such that for any $0\le s\le t\le 1$ and any $\sigma$-algebra $\scr G\subset \scr F_{\kappa_n(s)}$ the following holds 
		\begin{equation}
			\begin{aligned}
				\left\|\int_{s}^{t}\big(g(r,X^{(n)}_r)-g(r,X^{(n)}_{\kappa_n(r)})\big)\rrm dr \right\|_{L^2(\Omega)|\scr G} \le Cn^{-\big(\frac{1}{2}+\beta\wedge\frac{\eta}{\alpha}\wedge\frac{1}{2}\big)+\varepsilon}|t-s|^{\frac{1}{2}+\varepsilon}
			\end{aligned}
		\end{equation}
	\end{proposition}
	\begin{proof}
		Recall $\phi^{(n)}_t$ defined in \eqref{eqiso}, for fixed $0\le s\le t\le 1$, 
		\begin{equation*}
			\begin{aligned}
				&\int_{s}^{t}\big(g(r,X^{(n)}_r)-g(r,X^{(n)}_{\kappa_n(r)})\big)\rrm dr\\&=\int_{s}^{t}\big(g(r,L_r+\phi^{(n)}_r)-g(r,L_r+\phi^{(n)}_{\kappa_n(r)})\big)\rrm dr \\&\quad+\int_{s}^{t}\big(g(r,L_r+\phi^{(n)}_{\kappa_n(r)})-g(r,L_{\kappa_n(r)}+\phi^{(n)}_{\kappa_n(r)})\big)\rrm dr \\&=:I_1+I_2.
			\end{aligned}
		\end{equation*}
		Next, we estimate $I_1$ and $I_2$.  Since  $\phi^{(n)}_r$ and $\phi^{(n)}_{\kappa_n(r)}$ are bounded and adapted processes on $[0,1]$ as shown in Lemma \ref{X_t-x_s}, and for any $0\le s\le t\le 1$ and any $\varepsilon\in(0,1/2)$, Lemma \ref{lam3.2} yields the following estimates
		\begin{align}
			&\bb E_s\big|\phi^{(n)}_t-\bb E_s\phi^{(n)}_t\big|\le C|t-s|^{1+\beta\wedge\frac{\eta}{\alpha}-\varepsilon},\label{Lphi}\\&
			\bb E\big|\phi^{(n)}_{\kappa_n(t)}-\bb E_s\phi^{(n)}_{\kappa_n(t)}\big|\le \bb E\big[\bb E_s\big|\phi^{(n)}_{\kappa_n(t)}-\bb E_s\phi^{(n)}_{\kappa_n(t)}\big|\big]\le C|t-s|.
		\end{align}
		 Furthermore, the conditions $2\eta+\alpha>2$, $(\beta+1)\alpha+\eta>2$ ensure the applicability of \cite[Lemma 4.4]{butkovsky2024strongrateconvergenceeuler}, where we just take $\theta=\eta, \tau=1+\beta\wedge(\eta/\alpha)-\varepsilon,\varepsilon_0=1$ and $\gamma=1/2$. Therefore, by combining Lemmas \ref{lam3.2} and \ref{lem3.3}, we obtain the estimate
		\begin{align*}
			\|I_1\| _{L^2(\Omega)|\scr G}\le C|t-s|^{1+\frac{\eta-1}{\alpha}}n^{-1}+C|t-s|^{\frac{3}{2}+\frac{\eta-1}{\alpha}}n^{-\frac{1}{2}-\big(\beta\wedge\frac{\eta}{\alpha}\wedge\frac{1}{2}\big)+\varepsilon}.
		\end{align*}
		Next, we analysis $I_2$. We choose $\delta$ small enough so that $$\eta>1-\frac{\alpha}{2}+\delta\alpha.$$
		By equation \eqref{Lphi}, for any $0\le s\le r\le 1$ and $t\in[s,1]$, we know 
		\begin{align*}
			\bb E_s\big|\bb E_t\phi^{(n)}_r-\bb E_s \phi^{(n)}_r\big|=\bb E_s\big|\bb E_t\big[\phi^{(n)}_r-\bb E_s \phi^{(n)}_r\big]\big|\le \bb E_s\big|\phi^{(n)}_r-\bb E_s \phi^{(n)}_r\big|\le C|t-s|^{1+\beta\wedge\frac{\eta}{\alpha}-\delta}.
		\end{align*}
		Then, we obtain  $$\tilde \tau:=1+\beta\wedge\frac{\eta}{\alpha}-\delta>\frac{1}{2}+\frac{1}{\alpha}+\beta\wedge\frac{\eta}{\alpha}\wedge\frac{1}{2}-\frac{\eta}{\alpha}.$$
		We also know $\phi^{(n)}_{\kappa_n(t)}$ is $\scr{F}_{(\kappa_n(t)-\frac{1}{n})\vee1 }$-measurable for $t\in[0,1]$. Therefore, by \cite[Lemma 4.7]{butkovsky2024strongrateconvergenceeuler}, we know that 
		$$\|I_2\|_{L^2(\Omega)|\scr G}\le Cn^{-\big(\frac{1}{2}+\beta\wedge\frac{\eta}{\alpha}\wedge\frac{1}{2}\big)+2\varepsilon}|t-s|^{\frac{1}{2}+\varepsilon}.$$
		Finally, we get the result.
	\end{proof}

	\begin{proposition}\label{prop3.6}
		Let $\beta,\eta\in(0,1]$ and take $\varepsilon\in (0,1/2)$, and $L$ is a $d$-dimensional symmetric $\alpha$-stable process with $\alpha\in(1,2)$. Suppose that \eqref{co:AS} holds.
		Then
		for all $g_1\in \mathcal{C}_b^{\beta,\eta}$ and $g_2\in C([0,1]; C^\varpi_b(\bb R^d))$, where $\varpi\in(0,1)$, 
		$0\le S\le T\le 1$, $n\in\bb N$ and any $\sigma$-algebra $\scr G\subset\scr F_{\kappa_n(S)}$, there exists a constant $C$ such that
		\begin{align}
			\begin{split}
			\Bigl\|\int_S^T 	&\big(g_1(r,X^{(n)}_r)-g_1(r,X^{(n)}_{\kappa_n(r)})\big)g_2(r,X^{(n)}_r)\, \rrm dr\Bigr\|_{L^2(\Omega)|\scr G}\\
				&\le C\|g_1\|_{\mathcal{C}^{\alpha,\beta}_b}\|g_2\|_{C_b^{0,\varpi}} |T-S|^{1/2+\varepsilon}n^{-(1/2+\gamma-\varepsilon)},     
			\end{split}
		\end{align}
	where $\gamma=\beta\wedge(\eta/\alpha) \wedge (1/2)$.
	\end{proposition}
	\begin{proof}
			Define and divide the process as follows
		\begin{align*}
			\tilde{\al{A}}_{t}:=\int_0^t (g_1(r,X^{(n)}_r)-g_1(r,X^{(n)}_{\kappa_n(r)}))g_2(r,X^{(n)}_r)\, \rrm d r:= \tilde{\mathcal{A}}_{t}^1 + \tilde{\mathcal{A}}_{t}^2,
		\end{align*}
		where
		$$
		\tilde {\al A}_t^1:=\int_0^t  \big(g_1(r,L_r+\phi^{(n)}_r)-g_1(r,L_r+\phi^{(n)}_{\kappa_n(r)})\big)g_2(r,X^{(n)}_r)\,\rrm d r,
		$$
		and
		$$\tilde {\al A}_t^2:=\int_0^t  \big(g_1(r,L_r+\phi^{(n)}_{\kappa_n(r)})-g_1(r,L_{\kappa_n(r)}+\phi^{(n)}_{\kappa_n(r)})\big)g_2(r,X^{(n)}_r)\,\rrm d r,
		$$
		We will use conditional shifted sewing lemma Theorem \ref{condition sew} to prove this proposition. Therefore, for any $(s,t)\in\Delta_{S,T}$, let
		$$
		A_{s,t}^1:=\bb E_{s-(t-s)}\int_{s}^{t}\Big( \big(g_1(r,L_r+\bb E_{s-(t-s)}\phi^{(n)}_r)-g_1(r,L_{r}+\bb E_{s-(t-s)}\phi^{(n)}_{\kappa_n(r)})\big)g_2(r,X^{(n)}_{s-(t-s)})\Big)\rrm dr,
		$$
		and
		$$
		A_{s,t}^2:=\bb E_{s-(t-s)}\int_{s}^{t} \Big(\big(g_1(r,L_{\kappa_n(r)}+\bb E_{s-(t-s)}\phi^{(n)}_{\kappa_n(r)})-g_1(r,L_{r}+\bb E_{s-(t-s)}\phi^{(n)}_{\kappa_n(r)})\big)g_2(r,X^{(n)}_{s-(t-s)})\Big)\rrm dr,
		$$
		To simplify the notation, we define $s_1=s-(t-s), s_2=s-(u-s),s_3=s,s_4=u,s_5=t$, where $u=(s+t)/2$. Obviously, $s_1\le s_2\le s_3\le s_4\le s_5$.Using the regularity of $g_1$ and the boundedness of  $g_2$, we obtain the estimate
		\begin{align*}
			|A^1_{s,t}|&\le \bb E_{s_1}\int_{s_3}^{s_5} \big|\big(g_1(r,L_r+\bb E_{s_1}\phi^{(n)}_r)-g_1(r,L_{r}+\bb E_{s_1}\phi^{(n)}_{\kappa_n(r)})\big)g_2(r,X^{(n)}_{s_1})\big|\rrm dr\\&\le \left\|g_2 \right\|_\infty\|g_1\|_{C_b^{\beta,\eta}}\int_s^t|r-s|^{\frac{\eta-1}{\alpha}}\big|\bb E_{s_1}[\phi^{(n)}_r-\phi^{(n)}_{\kappa_n(r)}]\big|\rrm dr
		\end{align*}
		Using the regularity of $g_1$, this further yields
		\begin{align*}
			\left\|A^1_{s,t} \right\| _{L^2(\Omega)|\scr G}&\le \left\|g_2 \right\|_\infty\|g_1\|_{C_b^{\beta,\eta}}\int_s^t|r-s|^{\frac{\eta-1}{\alpha}}\big\|\bb E_{s_1}[\phi^{(n)}_r-\phi^{(n)}_{\kappa_n(r)}]\big\|_{L^2(\Omega)|\scr G}\rrm dr\\&\le C\left| t-s\right| ^{1+\frac{\eta-1}{\alpha}}n^{-1}.
		\end{align*}	
		Next, we estimate $\bb E_{s-(t-s)}\delta A_{(s,u,t)}$,
		\begin{align*}
			\delta A^1_{s,u,t}&=A^1_{s_3,s_5}-A^1_{s_3,s_4}-A^1_{s_4,s_5}\\
			&=\int_{s_3}^{s_4}g_2(r,X^{(n)}_{s_1})\times\\&\qquad\Big(\bb E_{s_1} \big[g_1(r,L_r+\bb E_{s_1}\phi^{(n)}_r)-g_1(r,L_r+\bb E_{s_1}\phi^{(n)}_{\kappa_n(r)})\big]-\bb E_{s_2}\big[g_1(r,L_r+\bb E_{s_2}\phi^{(n)}_r)-g_1(r,L_r+\bb E_{s_2}\phi^{(n)}_{\kappa_n(r)})\big]\Big)\rrm dr\\&\quad+\int_{s_3}^{s_4}\big(g_2(r,X^{(n)}_{s_1})-g_2(r,X^{(n)}_{s_2})\big)\bb E_{s_2}\big[g_1(r,L_r+\bb E_{s_2}\phi^{(n)}_r)-g_1(r,L_r+\bb E_{s_2}\phi^{(n)}_{\kappa_n(r)})\big]\rrm dr
			\\&\quad+ \int_{s_4}^{s_5}g_2(r,X^{(n)}_{s_1})\times\\&\qquad\Big(E_{s_1}\big[g_1(r,L_r+\bb E_{s_1}\phi^{(n)}_r)-g_1(r,L_r+\bb E_{s_1}\phi^{(n)}_{\kappa_n(r)})\big]-\bb E_{s_3}\big[g_1(r,L_r+\bb E_{s_3}\phi^{(n)}_r)+g_1(r,L_r+\bb E_{s_3}\phi^{(n)}_{\kappa_n(r)})\big]\Big)\rrm dr\\&\quad+ \int_{s_4}^{s_5}\big(g_2(r,X^{(n)}_{s_1})-g_2(r,X^{(n)}_{s_3})\big)\bb E_{s_3}\big[g_1(r,L_r+\bb E_{s_3}\phi^{(n)}_r)-g_1(r,L_r+\bb E_{s_3}\phi^{(n)}_{\kappa_n(r)})\big]\rrm dr\\&=:J_1+J_2+J_3+J_4.
		\end{align*}
		Since $g_2$ is bounded, we estimate $J_1$ and $J_2$ as follows
		\begin{align*}
			|J_1|&\le \int_{s_3}^{s_4}\Big|g_2(r,X^{(n)}_{s_1})\times\\&\qquad\Big(\bb E_{s_1} \big[g_1(r,L_r+\bb E_{s_1}\phi^{(n)}_r)-g_1(r,L_r+\bb E_{s_1}\phi^{(n)}_{\kappa_n(r)})\big]-\bb E_{s_2}\big[g_1(r,L_r+\bb E_{s_2}\phi^{(n)}_r)-g_1(r,L_r+\bb E_{s_2}\phi^{(n)}_{\kappa_n(r)})\big]\Big)\Big|\rrm dr\\&\le
			\|g_2\|_\infty\|g_1\|_{C_b^{\beta,\eta}}\int_{s_3}^{s_4}|r-s_2|^{\frac{\eta-1}{\alpha}}\bb E_{s_1}|\bb E_{s_1}[\phi^{(n)}_r -\phi^{(n)}_{\kappa_n(r)}]- \bb E_{s_2}[\phi^{(n)}_r -\phi^{(n)}_{\kappa_n(r)}]|\rrm dr\\&\quad+	\|g_2\|_\infty\|g_1\|_{C_b^{\beta,\eta}}\int_{s_3}^{s_4}|r-s_2|^{\frac{\eta-2}{\alpha}}\bb E_{s_1}|\phi^{(n)}_r -\phi^{(n)}_{\kappa_n(r)}|\bb E_{s_1}|\bb E_{s_1}\phi^{(n)}_r - \bb E_{s_2}\phi^{(n)}_r|\rrm dr.
		\end{align*} 
		And
		\begin{align*}
			|J_2|&\le \int_{s_3}^{s_4}\big|\big(g_2(r,X^{(n)}_{s_1})-g_2(r,X^{(n)}_{s_2})\big)\bb E_{s_2}\big[g_1(r,L_r+\bb E_{s_2}\phi^{(n)}_r)-g_1(r,L_r+\bb E_{s_2}\phi^{(n)}_{\kappa_n(r)})\big]\big|\rrm dr\\&\le 
			2\|g_2\|_\infty\|g_1\|_{C_b^{\beta,\eta}}\int_{s_3}^{s_4}|r-s_2|^{\frac{\eta-1}{\alpha}}\big|\bb E_{s_2}[\phi^{(n)}_r-\phi^{(n)}_{\kappa_n(r)}]\big|\rrm dr.
		\end{align*} 
	We apply similar techniques to $J_3$ and $J_4$. Then, using the conditional Minkowski's inequality, we obtain the estimate
		\begin{align*}
			&\Big\|\bb E_{s-(t-s)}\delta A^1_{(s,u,t)}\Big\|_{L^2(\Omega)|\scr G}\\&\le C\int_{s_3}^{s_4}|r-s_2|^{\frac{\eta-1}{\alpha}}\Big\|\bb E_{s_1}|\bb E_{s_1}[\phi^{(n)}_r -\phi^{(n)}_{\kappa_n(r)}]- \bb E_{s_2}[\phi^{(n)}_r -\phi^{(n)}_{\kappa_n(r)}]|\Big\|_{L^2(\Omega)|\scr G}\rrm dr\\&\quad+	C\int_{s_3}^{s_4}|r-s_2|^{\frac{\eta-2}{\alpha}}\Big\|\bb E_{s_1}|\phi^{(n)}_r -\phi^{(n)}_{\kappa_n(r)}|\Big\|_{L^2(\Omega)|\scr G}\Big\|\bb E_{s_1}|\bb E_{s_1}\phi^{(n)}_r - \bb E_{s_2}\phi^{(n)}_r|\Big\|_{L^2(\Omega)|\scr G}\rrm dr\\&\quad+	C\int_{s_3}^{s_4}|r-s_2|^{\frac{\eta-1}{\alpha}}\big\|\bb E_{s_2}[\phi^{(n)}_r-\phi^{(n)}_{\kappa_n(r)}]\big\|_{L^2(\Omega)|\scr G}\rrm dr\\&
			\quad+C\int_{s_4}^{s_5}|r-s_3|^{\frac{\eta-1}{\alpha}}\Big\|\bb E_{s_1}|\bb E_{s_1}[\phi^{(n)}_r -\phi^{(n)}_{\kappa_n(r)}]- \bb E_{s_3}[\phi^{(n)}_r -\phi^{(n)}_{\kappa_n(r)}]|\Big\|_{L^2(\Omega)|\scr G}\rrm dr\\&\quad+	C\int_{s_4}^{s_5}|r-s_3|^{\frac{\eta-2}{\alpha}}\Big\|\bb E_{s_1}|\phi^{(n)}_r -\phi^{(n)}_{\kappa_n(r)}|\Big\|_{L^2(\Omega)|\scr G}\Big\|\bb E_{s_1}|\bb E_{s_1}\phi^{(n)}_r - \bb E_{s_3}\phi^{(n)}_r|\Big\|_{L^2(\Omega)|\scr G}\rrm dr\\&\quad+	C\int_{s_4}^{s_5}|r-s_3|^{\frac{\eta-1}{\alpha}}\big\|\bb E_{s_3}[\phi^{(n)}_r-\phi^{(n)}_{\kappa_n(r)}]\big\|_{L^2(\Omega)|\scr G}\rrm dr\\&\le Cn^{-\frac{1}{2}-(\beta\wedge\frac{\eta}{\alpha})+\varepsilon}\left| t-s\right| ^{\frac{3}{2}+\frac{\eta-1}{\alpha}+\varepsilon}+Cn^{-1}\left| t-s\right|.
		\end{align*}
		Define partition points: $t_i=s+i\frac{t-s}{m}$, $i={1,2,\cdots,m}$. Next we will prove $\left|\tilde{\al A}_t^1-\tilde{\al A}_s^1-\sum_{i=1}^{m-1}A^1_{t_i,t_{i+1}} \right|\to 0$ when $m\to\infty$ in probability. Since,
		\begin{equation*}
			\begin{aligned}
				&\left|\tilde{\al A}_t^1-\tilde{\al A}_s^1-\sum_{i=1}^{m-1}A^1_{t_i,t_{i+1}} \right|\\& \le \left\|g_2 \right\|_{C_b^{0,\varpi}} \sum_{i=1}^{m-1}\int_{t_i}^{t_{i+1}}\big|g_1(r,L_r+\phi^{(n)}_r)-\bb E_{t_{i-1}}g_1(r,L_{r}+\bb E_{t_{i-1}}\phi^{(n)}_r) \big| \rrm dr\\&\quad+\left\|g_2 \right\|_{C_b^{0,\varpi}}\sum_{i=1}^{m-1}\int_{t_i}^{t_{i+1}}\big|g_1(r,L_{r}+\phi^{(n)}_{\kappa_n(r)})-\bb E_{t_{i-1}}g_1(r,L_{r}+\bb E_{t_{i-1}}\phi^{(n)}_{\kappa_n(r)}) \big| \rrm dr\\&\quad+\left\|g_2 \right\|_{C_b^{0,\varpi}}\int_{t_0}^{t_1}\big|g_1(r,L_r+\phi^{(n)}_r)-g_1(r,L_r+\phi^{(n)}_{\kappa_n(r)}) \big| \rrm dr\\&\quad+2\left\|g_1 \right\|_{C_b^{\beta,\eta}} \sum_{i=1}^{m-1}\int_{t_i}^{t_{i+1}}\big|g_2(r,X^{(n)}_r)-g_2(r,X^{(n)}_{t_{i-1}})\big|\rrm dr
				\\&=:I_1+I_2+I_3+I_4.
			\end{aligned}
		\end{equation*}
		Since $g_1(r,L_{t_{i-1}}+\bb E_{t_{i-1}}\phi_r^{(n)}$ is $\scr F_{t_{i-1}}$-measurable and we use Lemma \ref{lam3.2} and Jensen's inequality. We can deduce that for any $\varepsilon>0$
		\begin{equation*}
			\begin{aligned}
				&\bb E \big|g_1(r,L_r+\phi^{(n)}_r)-\bb E_{t_{i-1}}g_1(r,L_r+\bb E_{t_{i-1}}\phi^{(n)}_r) \big|\\&\le
				\bb E\big|g_1(r,L_r+\phi^{(n)}_r)-g_1(r,L_{t_{i-1}}+\bb E_{t_{i-1}}\phi^{(n)}_r)+g_1(r,L_{t_{i-1}}+\bb E_{t_{i-1}}\phi^{(n)}_r)-\bb E_{t_{i-1}}g_1(r,L_r+\bb E_{t_{i-1}}\phi^{(n)}_r) \big|\\&\le 
				2\bb E\big|g_1(r,L_r+\phi^{(n)}_r)-g_1(r,L_{t_{i-1}}+\bb E_{t_{i-1}}\phi^{(n)}_r)\big|\\&\le 
				C\left\|g_1\right\|_{C_b^{\beta,\eta}} \big(\bb E\big|\phi^{(n)}_r-\bb E_{t_{i-1}}\phi^{(n)}_r\big|^{\eta}+\bb E[|L_r-L_{t_{i-1}}|^{\eta}\wedge1]\big)\\&\le C\left|r-t_{i-1} \right| ^{\eta(1+\beta\wedge\frac{\eta}{\alpha}-\varepsilon)}+C|r-t_{i-1}|^{\frac{\eta}{\alpha}-\varepsilon}.
			\end{aligned}
		\end{equation*}
		Then it yields that $\bb E|I_1|\le C m ^{-\eta(1+\beta\wedge\frac{\eta}{\alpha}-\varepsilon)}+Cm^{-\frac{\eta}{\alpha}+\varepsilon}$.
		Similarly, we obtain the bound for $\bb E|I_2|\le C m ^{-\eta(1+\beta\wedge\frac{\eta}{\alpha}-\varepsilon)}+Cm^{-\frac{\eta}{\alpha}+\varepsilon}$. 
		This follows because when $\kappa_n(r)\le t_{i-1}$, the term $\phi^{(n)}_{\kappa_n(r)}$ is measurable. On the other hand, if $\kappa_n(r)> t_{i-1}$, we have the bound $\kappa_n(r)-t_{i-1}\le t_{i+1}-t_{i-1}$.
		Since $g_1$ is bounded, it yields that $$\bb E|I_3|\le Cm^{-1}.$$
		By Lemma \ref{X_t-x_s}, we have $$\bb E|I_4|\le Cm^{-\frac{\varpi}{\alpha}}.$$	
		According to Markov's inequality, for every $\delta>0$ we know 
		\begin{align*}
			\lim\limits_{m\to \infty}\left\{\left|\tilde{\al A}_t^1-\tilde{\al A}_s^1-\sum_{i=1}^{m-1}A^1_{t_i,t_{i+1}} \right|\ge \delta\right\}&\le\delta^{-1}\lim\limits_{m\to \infty}\bb E\left|\tilde{\al A}_t^1-\tilde{\al A}_s^1-\sum_{i=1}^{m-1}A^1_{t_i,t_{i+1}} \right| \\&
			=\delta^{-1}\lim\limits_{m\to \infty}\big(\bb E|I_1|+\bb E|I_2|+\bb E|I_3|+\bb E|I_4|\big)\\&=0.
		\end{align*}
		For $\tilde{\al A}^2$ and $A^2$, we apply the same method. Finally, we have 
		\begin{equation*}
			\begin{aligned}
				\big\|\tilde {\al A}_T-\tilde {\al A}_S\big\|_{L^2(\Omega)|\scr G}&\le 	\big\|\tilde {\al A}_T^1-\tilde {\al A}^1_S\big\|_{L^2(\Omega)|\scr G}+	\big\|\tilde {\al A}^2_T-\tilde {\al A}^2_S\big\|_{L^2(\Omega)|\scr G}\\&\le
				C\|g_1\|_{\mathcal{C}^{\alpha,\beta}_b}\|g_2\|_{C_b^{0,\varpi}} |T-S|^{1/2+\varepsilon}n^{-(1/2+\gamma-\varepsilon)}.
			\end{aligned}
		\end{equation*}
	\end{proof}
	\begin{proposition}\label{prop3.7}
		Let $\beta,\eta\in(0,1]$, $p\ge1$, and take $\varepsilon\in (0,1/2)$.  Suppose that \eqref{co:AS} holds.
		Then
		for all $g_1\in \mathcal{C}_b^{\beta,\eta}$ and $g_2\in C([0,1]; C^\varpi_b(\bb R^d))$, where $\varpi\in(0,1)$, 
		$0\le s\le t\le 1$, $n\in\bb N$ and any $p\ge1 $, there exists a constant $C$ such that
		\begin{align}
			\begin{split}
				&\Bigl\|\sup_{s \le u \le t}\Big|\int_s^u \big(g_1(r,X^{(n)}_r)-g_1(\kappa^\tau_n(r),X^{(n)}_{\kappa_n(r)})\big)g_2(r,X^{(n)}_r)\, \rrm dr\Big|\Bigr\|_{L^p(\Omega)}\le Cn^{-(1/2+\gamma-\varepsilon)}, 
			\end{split}
		\end{align}
		In particular, we have 
		\begin{align}
			\begin{split}
				&\Bigl\|\sup_{s \le u \le t}\Big|\int_s^u \big(g_1(r,X^{(n)}_{\kappa_n(r)})-g_1(\kappa^\tau_n(r),X^{(n)}_{\kappa_n(r)})\big)g_2(r,X^{(n)}_r)\, \rrm dr\Big|\Bigr\|_{L^p(\Omega)}\le Cn^{-(1/2+\gamma-\varepsilon)}, 
			\end{split}
		\end{align}
		where $\gamma=\beta\wedge(\eta/\alpha) \wedge (1/2)$.
	\end{proposition}
	\begin{proof}
		Define and divide the process as follows
		\begin{align*}
			\al{A}_{t}:=\int_0^t \big(g_1(r,X^{(n)}_r)-g_1(\kappa^\tau_n(r),X^{(n)}_{\kappa_n(r)})\big)g_2(r,X^{(n)}_r)\, \rrm d r:= \al{A}_{t}^1 + \al{A}_{t}^2,
		\end{align*}
		where
		$$
		 \al A_t^1:=\int_0^t  \big(g_1(r,X^{(n)}_r)-g_1(r,X^{(n)}_{\kappa_n(r)})\big)g_2(r,X^{(n)}_r)\,\rrm d r,
		$$
		and
		$$ \al A_t^2:=\int_0^t  \big(g_1(r,X^{(n)}_{\kappa_n(r)}-g_1(\kappa^\tau_n(r),X^{(n)}_{\kappa_n(r)})\big)g_2(r,X^{(n)}_r)\,\rrm d r,
		$$
		 According to Proposition \ref{prop3.6}, If there is a $k$ such that $0\le k/n\le s\le t\le 1$, it yields that 
		\begin{equation*}
			\begin{aligned}
				\bb E_{\frac{k}{n}}\big|\al A^1_t-\al A^1_s\big|\le \Big(\bb E_{\frac{k}{n}}\big|\al A^1_t-\al A^1_s\big|^2\Big)^\frac{1}{2}\le C n^{-(1/2+\gamma-\varepsilon)}
			\end{aligned}
		\end{equation*}
		Based on Weighted John-Nirenberg inequality \cite[Proposition 3.2]{butkovsky2024strongrateconvergenceeuler}, we obtain that 
		$$\big\|\sup_{s \le u \le t } |\al A^1_u-\al A^1_s|\big\|_{L^p(\Omega)|\scr F_{\frac{k}{n}}}\le C n^{-(1/2+\gamma-\varepsilon)}.$$
		Take expectation between the inequality, we have 
		$$\big\|\sup_{s \le u \le t } |\al A^1_u-\al A^1_s|\big\|_{L^p(\Omega)}=\big\|\big\|\sup_{s \le u \le t } |\al A^1_u-\al A^1_s|\big\|_{L^p(\Omega)|\scr F_{\frac{k}{n}}}\big\|_{L^p(\Omega)}\le C n^{-(1/2+\gamma-\varepsilon)}.$$
		Then by applying Kolmogorov continuity theorem to Lemma \ref{lem3.1}, it yields  

		$$\big\|\sup_{s \le u \le t } |\al A^2_u-\al A^2_s|\big\|_{L^p(\Omega)}\le C n^{-(1/2+\beta-\varepsilon)}.$$
		Based on the above inequalities, we obtain the following estimate
		\begin{equation*}
			\begin{aligned}
				\big\|\sup_{s \le u \le t}|\al A_u-\al A_s| \big\|_{L^p(\Omega)}&\le \big\|\sup_{s \le u \le t}|\al A^1_t-\al A^1_s| \big\|_{L^p(\Omega)}+\big\|\sup_{s \le u \le t}|\al A^2_t-\al A^2_s| \big\|_{L^p(\Omega)}\\&\le 
				 C n^{-(1/2+\gamma-\varepsilon)}
			\end{aligned}
		\end{equation*}
	\end{proof}

	\section{Proof of Theorem \ref{theo2.3}} \label{sec:mainthm}
	
	 We consider the backward parabolic equation \eqref{kolmo}
	\begin{align*}
		\partial_t u -\lambda u + \nabla u \cdot b + \al L u+b=0 \text{ on } [0,1] \times \bb R^d,~ u(1,x)=0.
	\end{align*}
	We know $L_t$ is a $d$-dimensional symmetric $\alpha$-stbale process. According to \eqref{eq,unorm}, we can choose $\lambda>0$ large enough such that 
    \begin{equation}\label{nabla,u}
        \sup_{0\le t\le 1}\left\|\nabla u(t,\cdot) \right\|_{C_b^\varpi}<\frac{1}{2},
    \end{equation}
	
	For notation simplicity, set 
	$$M_t^{\lambda}:=X_t-X_t^{(n)}+u(t,X_t)-u(t,X_t^{(n)}).$$
	Applying It\^o's formula, we obtain
	\begin{align*}
		\rrm dM_t^{\lambda}&=b(t,X_t)\rrm dt -b(\kappa_n^\tau(t),X_{\kappa_n(t)}^{(n)})\rrm dt+ \partial_tu(t,X_t)\rrm dt- \partial_tu(t,X^{(n)}_t)\rrm dt\\&\quad+ b(t,X_t)\nabla u(t,X_t)\rrm dt-b(\kappa_n^\tau(t),X_{\kappa_n(t)}^{(n)})\nabla u(t,X^{(n)}_t)\rrm dt\\&\quad+\int_{\{|z|\le1\}}u(t,X_{t^-}+z)-u(t,X_{t^-})\tilde N(\rrm dt,\rrm dz)+\int_{\{|z|>1\}}|u(t,X_{t^-}+z)-u(t,X_{t^-}) N(\rrm dt,\rrm dz)\\&\quad+\int_{\{|z|\le1\}}\big(u(t,X_{t}+z)-u(t,X_{t})-\left\langle z,\nabla u(t,X_{t}) \right\rangle \big)\upsilon(\rrm dz)\rrm dt\\&\quad+\int_{\{|z|\le1\}}u(t,X^{(n)}_{t^-}+z)-u(t,X^{(n)}_{t^-})\tilde N(\rrm dt,\rrm dz)+\int_{\{|z|>1\}}u(t,X^{(n)}_{t^-}+z)-u(t,X^{(n)}_{t^-}) N(\rrm dt,\rrm dz)\\&\quad+\int_{\{|z|\le1\}}\big(u(t,X^{(n)}_{t}+z)-u(t,X^{(n)}_{t})-\big\langle z,\nabla u(t,X^{(n)}_{t}) \big\rangle \big)\upsilon(\rrm dz)\rrm dt\\&=
		\big(b(t,X_t^{(n)})-b(\kappa_n^\tau(t),X_{\kappa_n(t)}^{(n)})\big)\rrm dt+\lambda\big(u(t,X_t)-u(t,X_t^{(n)})\big)\rrm dt\\&\quad+ \nabla u(t,X_t^{(n)})\big(b(t,X_t^{(n)})-b(\kappa_n^\tau(t),X_{\kappa_n(t)}^{(n)})\big)\rrm dt+ \int_{\bb R^d\setminus \{0\}}H(t,z)\tilde N(\rrm dt,\rrm dz),
	\end{align*}
	where $H(t,z)=u(t,X_{t^-}+z)-u(t,X_{t^-})-\big(u(t,X^{(n)}_{t^-}+z)-u(t,X^{(n)}_{t^-})\big)$.
	
	When $p\ge1$, we  apply the convexity inequality (also known as the $C_r$-inequality) $(a+b)^p\le 2^{p-1}a^p+2^{p-1}b^p$ and \eqref{nabla,u} such that
	\begin{equation*}
		\begin{aligned}
				\big| X_t-X_t^{(n)}\big| ^p&\le 2^{p-1}\left|M^{\lambda}_t \right|^p+2^{p-1}\big|u(t,X_t)-u(t,X_t^{(n)}) \big|  ^p\\&\le 2^{p-1}\left|M^{\lambda}_t \right|^p+\frac{1}{2}\big|X_t-X_t^{(n)} \big|  ^p.
		\end{aligned}
	\end{equation*} 
	Rearranging terms yields
		$$\big| X_t-X_t^{(n)}\big| ^p\le 2^{p+1}\left|M^\lambda_t \right|^p.$$	
	To proceed, we aim to estimate the expectation $\bb E\left|M_t \right|^p$. By applying standard inequalities, we obtain
	\begin{align*}
		\left|M^\lambda_t \right|^p&\le 4^{p-1}\Big|\int_{0}^{t}b(s,X_s^{(n)})-b(\kappa_n^\tau(s),X_{\kappa_n(s)}^{(n)})\rrm ds\Big|^p+4^{p-1}\lambda^p\big(\int_{0}^{t}\big|u(s,X_s)-u(s,X_s^{(n)})\big|\rrm ds\big)^p\\&\quad+ 4^{p-1}\Big|\int_{0}^{t}\nabla u(s,X_s^{(n)})\big(b(s,X_s^{(n)})-b(\kappa_n^\tau(s),X_{\kappa_n(s)}^{(n)})\big)\rrm ds\Big|^p+ 4^{p-1}\Big|\int_{0}^{t}\int_{\bb R^d\setminus \{0\}}H(s,z)\tilde N(\rrm ds,\rrm dz)\Big|^p\\&=:4^{p-1}\Theta_1(t)+4^{p-1}\lambda\Theta_2(t)+4^{p-1}\Theta_3(t)+4^{p-1}\Theta_4(t).
	\end{align*}
    Next, we will give the estimation of $\Theta_1(t),\Theta_2(t),\Theta_3(t),\Theta_4(t)$. To estimate $\Theta_4$, we decompose it into small- and large-jump contributions
	\begin{equation*}
		\begin{aligned}
			\bb E\Big[\sup_{0\le s\le t}\big|\Theta_4(s)\big|\Big]&\le C\bb E\Big[\sup_{0\le s\le t}\Big|\int_{0}^{s}\int_{|z|>1} H(r,z)\tilde N(\rrm dr,\rrm dz)\Big| ^p\Big]\\&\quad+
			C\bb E\Big[\sup_{0\le s\le t}\Big|\int_{0}^{s}\int_{|z|\le1} H(r,z)\tilde N(\rrm dr,\rrm dz)\Big|^p\Big]=:\Theta_{4,1}+\Theta_{4,2}.
		\end{aligned}
	\end{equation*}
	We observe that
	\begin{align*}
		|H(t,z)|\le |u(t,X_{t^-}+z)-u(t,X^{(n)}_{t^-}+z)|+|u(t,X_{t^-})-u(t,X^{(n)}_{t^-})|\le C|X_{t^-}-X^{(n)}_{t^-}|.
	\end{align*}
	Moreover, by \cite[Lemma 4.1]{MR2945756}, we also have for $|z|\le1$, 
	\begin{align}\label{eq,esti}
		|H(t,z)|\le C\left\|u \right\| _{C_b^{0,\alpha+\eta}}|X_{t^-}-X^{(n)}_{t^-}||z|^{\alpha+\eta-1},\quad\text{when $|z|\le1$}.
	\end{align}
	Using the boundedness of $u$, the properties of $\upsilon$ and Kunita’s inequality \cite[Theorem 4.4.23]{levy}, we obtain for $p\ge2$
	\begin{equation}\label{theta42}
		\begin{aligned}
			\Theta_{4,2}&\le C\bb E\Big[\big(\int_{0}^{t}\int_{|z|>1}\big|X_{s}-X^{(n)}_{s}\big|^2\upsilon(\rrm dz)\rrm ds\big)^{\frac{p}{2}}\Big]+C\bb E\Big[\int_{0}^{t}\int_{|z|>1}\big|X_{s}-X^{(n)}_{s}\big|^p\upsilon(\rrm dz)\rrm ds\Big]\\&\le
			C\big(\upsilon(|z|>1)^{\frac{p}{2}}+\upsilon(|z|>1)\big)\int_{0}^{t}\bb E\big|X_{s}-X^{(n)}_{s}\big|^p\rrm ds\\&\le C\int_{0}^{t}\bb E\big|X_{s}-X^{(n)}_{s}\big|^p\rrm ds.
		\end{aligned}
	\end{equation}
	For the small-jump component, by using \eqref{eq,esti}, we obtain
	\begin{equation}\label{theta41}
		\begin{aligned}
			\Theta_{4,1}&\le C\bb E\Big[\big(\int_{0}^{t}\int_{|z|\le1}\big|X_{t^-}-X^{(n)}_{t^-}\big|^2|z|^{2(\alpha+\eta-1)}\upsilon(\rrm dz)\rrm ds\big)^{\frac{p}{2}}\Big]\\&\quad+C\bb E\Big[\int_{0}^{t}\int_{|z|\le1}\big|X_{t^-}-X^{(n)}_{t^-}\big|^p|z|^{p(\alpha+\eta-1)}\upsilon(\rrm dz)\rrm ds\Big]\\&\le C\int_{0}^{t}\bb E\big|X_{s^-}-X^{(n)}_{s^-}\big|^p\rrm ds.
		\end{aligned}
	\end{equation}
    In summary, we get the estimation of $\Theta_4$,
    \begin{equation}
    \bb E\Big[\sup_{0 \le s \le t}|\Theta_4(s)|\Big]\le C\int_0^t\bb E|X_{s^-}-X^{(n)}_{s^-}|^p\rrm ds.
    \end{equation}
	Because $u\in C([0,1];C^{1+\gamma}_b(\bb R^d))$ in Theorem \ref{backward}, it yields that 
	\begin{equation}\label{theta2}
		\begin{aligned}
			\bb E\Big[\sup_{0 \le s \le t}\big(\int_{0}^{s}\big|u(r,X_r)-u(r,X^{(n)}_r)\big|\rrm dr\big)^p\Big]\le C\int_{0}^{t}\bb E\big|X_s-X^{(n)}_s\big|^p\rrm ds.
		\end{aligned}
	\end{equation}

	To estimate the last two remaining terms, we invoke Proposition \ref{prop3.7} with $g_2=1$. Then it yields that 
		\begin{equation}\label{theta1}
		\begin{aligned}
			\bb E\Big[\sup_{0 \le s \le t}\Theta_1(s)\Big]&\le C\Big(\Big\|\sup_{0 \le s \le t}\big|\int_{0}^{s}b(r,X_{\kappa_n(r)}^{(n)})-b(\kappa_n^\tau(r),X_{\kappa_n(r)}^{(n)})\rrm dr \big|\Big\|^p_{L^p(\Omega)}\\&\qquad
			+\Big\|\sup_{0 \le s \le t}\big|\int_{0}^{s}b(r,X_{r}^{(n)})-b(r,X_{\kappa_n(r)}^{(n)})\rrm dr\big| \Big\|^p_{L^p(\Omega)} \Big)\\&\le Cn^{-(1/2+\gamma+\varepsilon)}.
		\end{aligned}
	\end{equation}

	Similarly, for $\Theta_3(t)$, we have
	\begin{equation}\label{theta3}
		\begin{aligned}
			\bb E\Big[\sup_{0 \le s \le t}\Theta_3(s)\Big]&\le C\Big(\Big\|\sup_{0 \le s \le t}\big|\int_{0}^{s}\nabla u(r,X_r^{(n)})\big(b(r,X_{\kappa_n(r)}^{(n)})-b(\kappa_n^\tau(r),X_{\kappa_n(r)}^{(n)})\big)\rrm dr\big| \Big\|^p_{L^p(\Omega)}\\&\qquad
			+\Big\|\sup_{0 \le s \le t}\big|\int_{0}^{s}\nabla u(r,X_r^{(n)})\big(b(r,X_r^{(n)})-b(r,X_{\kappa_n(r)}^{(n)})\big)\rrm dr\big| \Big\|^p_{L^p(\Omega)} \Big)\\&\le Cn^{-(1/2+\gamma+\varepsilon)p}.
		\end{aligned}
	\end{equation}
	Combining inequalities \eqref{theta42}-–\eqref{theta3}, we obtain
	
	\begin{equation}
		\begin{aligned}
			\bb E\Big[\sup_{0 \le s \le t}\big|X_t-X^{(n)}_t\big|^p\Big]&\le C\bb E\Big[\sup_{0 \le s \le t}\big|M^\lambda_t\big|^p\Big]\\&\le C\int_{0}^{t}\bb{E}\Big[\sup_{0 \le s \le t}\big|X_r-X^{(n)}_r\big|^p\Big] \rrm ds+Cn^{-(1/2+\gamma+\varepsilon)p}.
		\end{aligned}
	\end{equation}
	Applying Gronwall’s inequality yields

	\begin{align*}
		\bb E\Big(\sup_{0\le s \le 1}\left|X_s-X_s^{(n)} \right|^p \Big)\le Cn^{-(1/2+\gamma-\varepsilon)p},
	\end{align*}
	When $p\in[1,2)$, the $L^p$-norm is controlled by the $L^2$-norm due to Jensen’s inequality. Therefore, the convergence holds for all $p\ge1 $, and the result follows.
	\qed

    \section{Numerical experiments}\label{sec:numerical}

For numerical experiments, we consider the scalar SDE \eqref{main SDE} over $[0,1]$ with the transform  $Y_t = X_t - L_t$, where  $\alpha\in \{1.25,1.5,1.75\}.$
The random differential equation would be
\begin{equation}
    \mathrm{d}Y_t = b(t, Y_t+L_t)\,\mathrm{d}t, \quad Y_0 = 0. 
\end{equation}
The EM simulation in  \eqref{eqn:emcontinuous} (resp. randomised EM simulation in \eqref{con EM SDE}) for $X_t$ is equivalent to the Euler (resp. randomised Euler) simulation for $Y_t$. The randomised algorithm can be found in Algorithm \ref{alg}.

To compare the performance of the EM method and the randomised EM method, we design a series of experiments with various coefficients $b$. A high-precision reference solution is computed using the randomised EM method with a finest stepsize \( h_{\text{ref}} = 2^{-17} \). The performances are evaluated via 300 independent simulations for each \( h = 2^{-l} \), \( l = 5, \ldots, 10 \). The error is quantified as the root mean square of the maximum absolute deviation across the time interval:
\[
\text{Error} = \left( \frac{1}{N} \sum_{i=1}^{N} \left( \max_{j} \left| Y_j^{\text{num},(i)} - Y_{j}^{\text{ref}, (i)} \right| \right)^2 \right)^{1/2},
\]
where $N = 300$ is the number in simulations, $Y_j^{\text{num},(i)}$
is the simulation value from either EM or randomised EM at the $j$-th time step in the $i$-th simulation, and $Y_{j}^{\text{ref}, (i)}$ is the corresponding reference solution value. The experiment code can be found in \href{https://github.com/iooiooiooiooi/Randomised-EM-Method-for-SDEs-with-Holder-Continuous-Drift-Coeffient-Driven-by--Stable-Levy-Process}{this GitHub repository}.

\vspace{1ex}
\vspace{1ex}
\vspace{1ex}
\begin{algorithm}[H]
\caption{Randomised Euler--Maruyama Method under Lévy Increments\label{alg}}
\KwIn{Initial value $Y_0$, stepsize $h$, final time $T$, drift function $b$, stable index $\alpha$}
\KwOut{Approximate solution $\{Y_k\}$ at times $t_k = kh$}

Set $N = T/h$;

Generate Lévy increments: $\Delta L_k \sim \text{Stable}(\alpha)$ for $k = 0, \dots, N-1$\;

Construct Lévy path: $L_0 = 0$, and for $k \geq 1$, set $L_k = L_{k-1} + \Delta L_{k-1}$\;

\For{$k = 0$ to $N-1$}{
    Generate $\theta_k \sim \mathcal{U}(0,1)$\;
    Set $t_k^\theta = t_k + h \cdot \theta_k$\;
    
    Update $Y_{k+1} = Y_k + h \cdot b(t_k^\theta, Y_k + L_k)$\;
}
\Return $\{Y_k\}$
\end{algorithm}

\subsection{Example 1: the Weierstrass  function}

Inspired by \cite{ellinger2025optimalerrorratesstrong},  the drift function \( b(t,x) \) is defined as a product of two Weierstrass-type functions - one in time and one in space:
\begin{equation}\label{eqn:bexample1}
    b(t,x) := W(t) \cdot \mu(x),
\end{equation}
where
\begin{align}\label{eqn:Wfunction}
\begin{split}
        W(t) &= \sum_{k=0}^{N_{W}-1} a^k \cos(\pi b^k t), \quad t \in [0,1],\\
    \mu(x) &= \sum_{j=0}^{{N_{\mu}-1}} 2^{-\gamma j} \cos(2^j \pi x), \quad x \in \mathbb{R}.
\end{split}
\end{align}
 Both $W(t)$ and $\mu(x)$ are truncated variants of the classical Weierstrass function, truncated to $N_{W}$ and $N_{\mu}$ terms respectively.
 It is well-known that these functions with infinitely many terms are continuous but nowhere differentiable and are known to be H\"older continuous for suitable choice of $a,b$ and $\gamma$.  Thus, the drift \( b(t,x) \) inherits the irregularity in both time and space. In our simulations, we use the parameters \( a = 0.5 \), \( b = 12 \), \( N_{W} = 25 \), \( \gamma = 0.5 \), and \( N_{\mu} = 25 \) in \eqref{eqn:Wfunction} so that $b$ is (approximately) $(\ln{2}/\ln{12})$-H\"older continuous in time and bounded $1/2$-H\"older continuous in space. It is easy to verify the triplet $(\alpha,\beta,\eta)$ satisfy the condition \eqref{co:AS}.

As shown in Figure~\ref{fig:weierstrass},the randomised EM method consistently outperforms the classical EM method in terms of both accuracy and convergence order. Specifically, when $\alpha=1.5$, the estimated slope of the convergence line is approximately $0.86$ for the randomised method, compared to $0.56$ for the classical method. This indicates that the randomisation technique improves the convergence behaviour under irregular drift functions. While both methods use the same number of simulations, the randomised version achieves significantly lower error, especially for smaller step sizes.

\begin{figure}[H]
    \centering
    \begin{subfigure}[b]{0.31\textwidth}
        \centering
        \includegraphics[width=\textwidth]{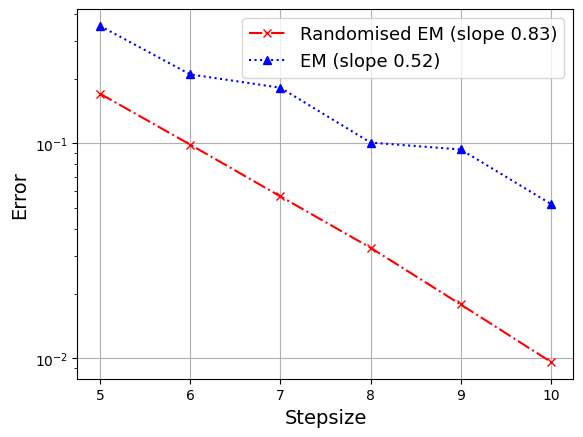}
        \caption{$\alpha=1.25$.}
        \label{fig:sub1}
    \end{subfigure}
    \hfill
    \begin{subfigure}[b]{0.31\textwidth}
        \centering
        \includegraphics[width=\textwidth]{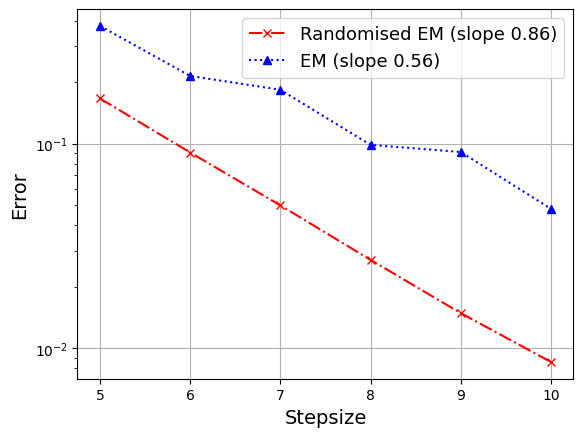}
        \caption{$\alpha=1.5$.}
        \label{fig:sub2}
    \end{subfigure}
    \hfill
    \begin{subfigure}[b]{0.31\textwidth}
        \centering
        \includegraphics[width=\textwidth]{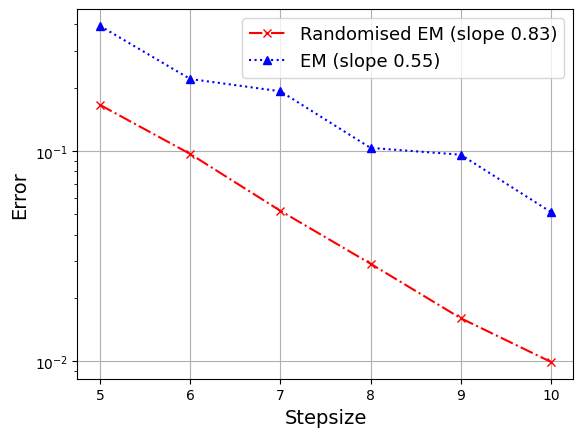}
        \caption{$\alpha=1.75$.}
        \label{fig:sub3}
    \end{subfigure}
    \caption{Semi-log plots of error vs.\ stepsize for Example 1 (the Weierstrass function).}
    \label{fig:weierstrass}
\end{figure}

\subsection{Example 2: the jump-type time-dependent function}

Inspired by \cite[Example 6.2]{Krusewu2017}, we consider a time-dependent drift with multiple discontinuities, defined via a weighted sum of sign functions:
\begin{equation}
    g(t) = \sum_{k=1}^{16} c_k \, \mathrm{sign}\left( \frac{k}{16}T - t \right),
    \label{eq:jump_drift}
\end{equation}
where \( T = 1 \) is the terminal time, and the coefficients \( \{c_k\} \) decrease linearly from \(-1\) to \(-0.05\), i.e., 
\[
    c_k = -1 + \frac{0.95(k-1)}{15}.
\]
Take
\begin{equation}
    b(t,x) := g(t) x
\end{equation}
so that the drift is constructed to contain multiple jump discontinuities, serving as a test case to evaluate the performance of numerical solvers for SDEs with non-smooth drift, which is not covered by the assumptions in this paper. In particular, the resulting drift exhibits jump discontinuities at 16 equidistant time points    $t_k = \frac{k}{16}T, \quad k = 1, 2, \ldots, 16.$

It can be observed from Figure~\ref{fig:jump} that across all three values of $\alpha$, although the classical EM method achieves a slightly higher convergence rate, the randomised EM method consistently achieves smaller errors across all tested stepsizes. This suggests that, for SDEs with jump discontinuities in the drift, the randomised EM may provide better practical accuracy despite a comparable theoretical rate.

\begin{figure}[H]
    \centering
    \begin{subfigure}[b]{0.31\textwidth}
        \centering
        \includegraphics[width=\textwidth]{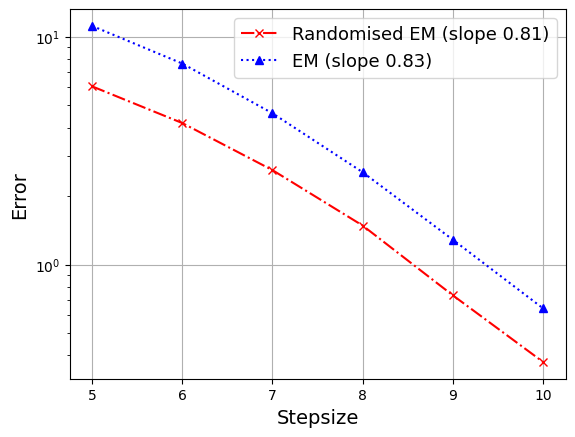}
        \caption{$\alpha=1.25$.}
        \label{fig:sub1}
    \end{subfigure}
    \hfill
    \begin{subfigure}[b]{0.31\textwidth}
        \centering
        \includegraphics[width=\textwidth]{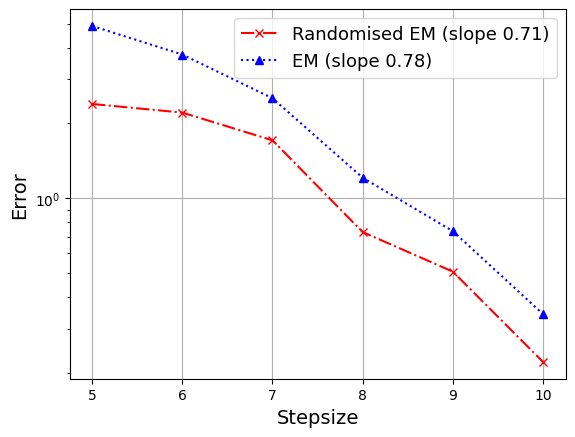}
        \caption{$\alpha=1.5$.}
        \label{fig:sub2}
    \end{subfigure}
    \hfill
    \begin{subfigure}[b]{0.31\textwidth}
        \centering
        \includegraphics[width=\textwidth]{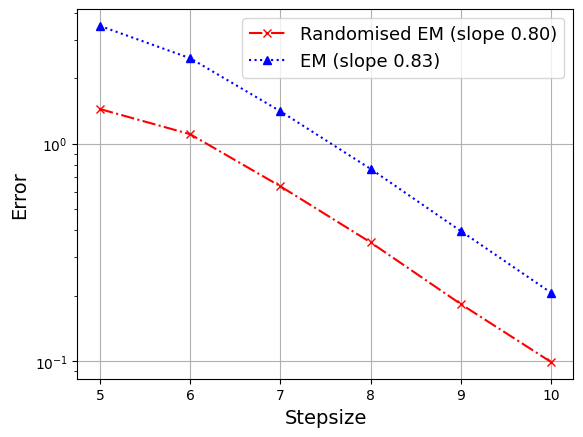}
        \caption{$\alpha=1.75$.}
        \label{fig:sub3}
    \end{subfigure}
    \caption{Semi-log plots of error vs.\ stepsize for Example 2 (the jump-type drift).}
    \label{fig:jump}
\end{figure}

\subsection{Example 3: the piecewise root function}

The drift function in this example is given by
\begin{equation}
    b(t, x) := g(t)x + \sqrt{\max(x, 10^{-12})},   \quad x \in \mathbb{R}, 
    \label{eq:root_drift}
\end{equation}
where \( g(t) \) is a piecewise constant function defined in ~\eqref{eq:jump_drift}. This drift function combines a discontinuous time-dependent component \( g(t) \) with a nonlinear spatial term involving the square root. To ensure numerical stability when the state variable takes non-positive values, we implement the square-root term as \( \sqrt{\max(x, 10^{-12})} \). While this avoids undefined evaluations, the drift remains non-Lipschitz in space due to the sharp nonlinearity near \( x = 0 \), and is discontinuous in time due to the piecewise nature of \( g(t) \).

Figure~\ref{fig:root} shows the convergence performance of both the standard EM method and the randomised EM method applied to the drift function ~\eqref{eq:root_drift}. We observe that both methods achieve a similar empirical convergence rate of approximately \( \mathcal{O}(h^{0.7 \sim 0.8}) \). However, the randomised EM method consistently has lower errors across all tested step sizes, demonstrating superior accuracy. This indicates its enhanced robustness when dealing with drift functions that are discontinuous in time and nonlinear in space.

\begin{figure}[H]
    \centering
    \begin{subfigure}[b]{0.31\textwidth}
        \centering
        \includegraphics[width=\textwidth]{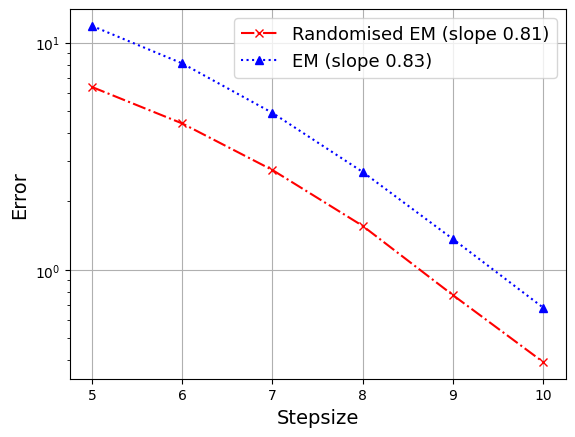}
        \caption{$\alpha=1.25$.}
        \label{fig:sub1}
    \end{subfigure}
    \hfill
    \begin{subfigure}[b]{0.31\textwidth}
        \centering
        \includegraphics[width=\textwidth]{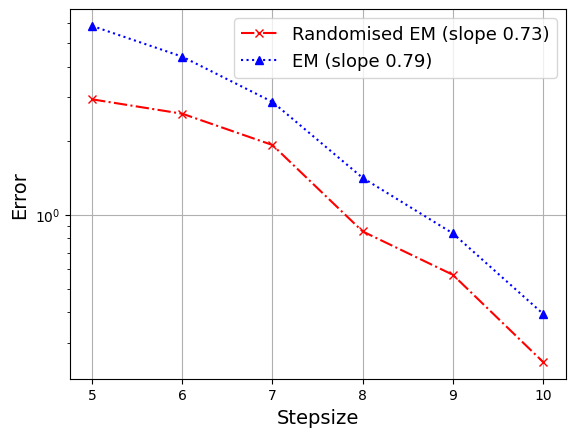}
        \caption{$\alpha=1.5$.}
        \label{fig:sub2}
    \end{subfigure}
    \hfill
    \begin{subfigure}[b]{0.31\textwidth}
        \centering
        \includegraphics[width=\textwidth]{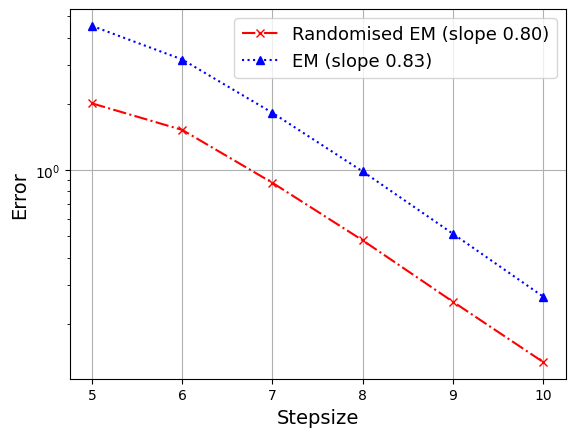}
        \caption{$\alpha=1.75.$}
        \label{fig:sub3}
    \end{subfigure}
    \caption{Semi-log plots of error vs.\ stepsize for Example 3 (the piecewise root function).}
    \label{fig:root}
\end{figure}

\subsection{Example 4: the temporal highly-oscillated function }
We construct a drift function of the form
\begin{equation}
    b(t, x) := \min\left( |x|^{\text{order}_1}, \text{lbd} \right) \cdot \left| \sin(w t) \right|^{1/20} - t^{\text{order}_2},
\end{equation}
where \( \text{lbd} > 0 \) is a positive lower bound, \( \text{order}_1 \in (0, 1) \) controls the growth in \( x \), and \( \text{order}_2 \in (0, 1) \) determines the regularity in time. The function introduces a combination of spatial saturation and temporal oscillation, modulated by a sine function.

In our study, we set the parameters as follows:

\[
\text{lbd}  = 20, \quad w = 2^8 \pi, \quad \text{order}_{1} = \frac{1}{5}, \quad \text{order}_{2} = \frac{1}{10}.\]

This drift function combines spatial non-Lipschitz growth, temporal oscillations, and a saturation mechanism. The spatial component \( \min(|x|^{\text{order}_1}, \text{lbd}) \) limits the growth of the state variable, while the oscillatory factor \( |\sin(wt)|^{1/20} \) and the decay term \( t^{\text{order}_2} \) introduce irregularities in time. Such construction poses significant challenges for standard numerical solvers and serves as a rigorous test case for evaluating robustness under non-smooth dynamics.

As illustrated in Figure~\ref{fig:DIY}, the classical EM method fails to converge under this challenging drift, exhibiting a very low convergence rate of 0.39. In contrast, the randomised EM method achieves a significantly better convergence rate of approximately 0.88, with consistently lower errors across all tested step sizes. This demonstrates the robustness of the randomised method in the presence of spatially saturated, temporally oscillatory and non-Lipschitz drift terms.

\begin{figure}[H]
    \centering
    \begin{subfigure}[b]{0.31\textwidth}
        \centering
        \includegraphics[width=\textwidth]{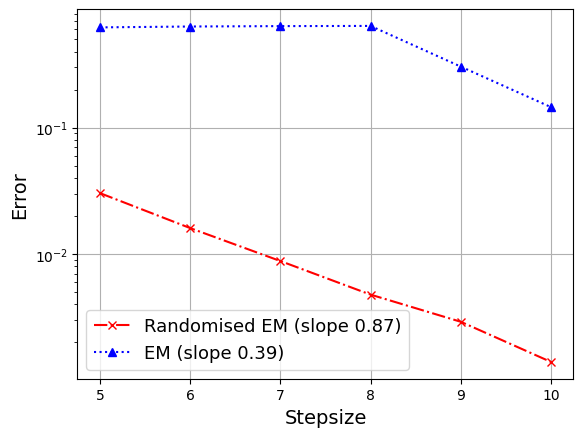}
        \caption{$\alpha=1.25$.}
        \label{fig:sub1}
    \end{subfigure}
    \hfill
    \begin{subfigure}[b]{0.31\textwidth}
        \centering
        \includegraphics[width=\textwidth]{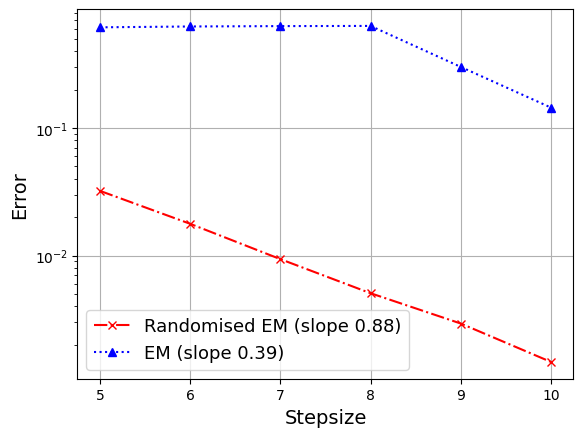}
        \caption{$\alpha=1.5$.}
        \label{fig:sub2}
    \end{subfigure}
    \hfill
    \begin{subfigure}[b]{0.31\textwidth}
        \centering
        \includegraphics[width=\textwidth]{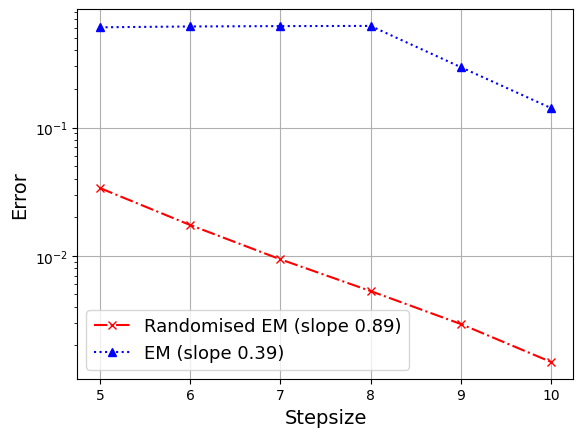}
        \caption{$\alpha=1.75.$}
        \label{fig:sub3}
    \end{subfigure}
    \caption{Semi-log plots of error vs.\ stepsize for Example 4 (the temporal highly-oscillated function). }
    \label{fig:DIY}
\end{figure}
 \section*{Acknowledgments}
YW would like to acknowledge the support of the Royal Society through the International Exchanges scheme IES\textbackslash R3\textbackslash 233115.

\section*{Conflict of interest}
None of the authors have a conflict of interest to disclose.

\bibliographystyle{plain} 
\bibliography{reference.bib}
\end{document}